\documentclass[a4paper,12pt]{amsart}
\usepackage{amsmath,amssymb,amsthm,leftidx}

\newcommand{\Id}{\mathbf{1}}

\renewcommand{\>}{\rangle}
\newcommand{\<}{\langle}

\newcommand{\BB}{\operatorname{\mathcal{B}}}
\newcommand{\EE}{\operatorname{\mathcal{E}}}
\newcommand{\LL}{\operatorname{\mathcal{L}}}

\newcommand{\PP}{\operatorname{\mathfrak{P}}}

\newcommand{\eps}{\epsilon}

\newcommand{\CR}{\bar{\partial}}

\newcommand{\diag}{\operatorname{diag}}

\newcommand{\del}{\partial}

\newcommand{\IC}{\operatorname{\mathbb{C}}}
\newcommand{\IZ}{\operatorname{\mathbb{Z}}}

\newcommand{\IR}{\operatorname{\mathbb{R}}}
\newcommand{\IN}{\operatorname{\mathbb{N}}}
\newcommand{\IH}{\operatorname{\mathbb{H}}}
\newcommand{\IF}{\operatorname{\mathbb{F}}}

\renewcommand{\LL}{\operatorname{\mathcal{L}}}

\newtheorem{theorem}{Theorem}[section]
\newtheorem{proposition}[theorem]{Proposition}

\newtheorem{definition}[theorem]{Definition}
\newtheorem{lemma}[theorem]{Lemma}
\newtheorem{corollary}[theorem]{Corollary}
\newtheorem{remark}[theorem]{Remark}

\title{Infinite-dimensional\\ symplectic non-squeezing\\ using non-standard analysis}
\author{Oliver Fabert}
\thanks{O. Fabert, VU Amsterdam, The Netherlands. Email: o.fabert@vu.nl}
\begin{document}
\maketitle

\begin{abstract}
We prove a non-squeezing result for infinite-dimensional Hamiltonian flows using non-standard model theory. For this we prove the existence of a corresponding family of pseudoholomorphic spheres and characterize the maximal time in terms of a limiting behaviour for these curves. While our proof is based on the finite-dimensional results from Gromov's original proof, we do not ask for any prior knowledge of non-standard model theory.
\end{abstract}

\tableofcontents
\markboth{O. Fabert}{Infinite-dimensional non-squeezing} 

\section{Introduction and summary}

It is known that many evolutionary partial differential equations, such as the Kortweg-de Vries equation, the nonlinear Schr\"odinger equation, and the nonlinear wave equation, belong to the class of so-called \emph{Hamiltonian partial differential equations}, where we refer to \cite{K} for definitions, statements and further references. This means that they can be written in the form $\del_t u=X^H_t(u)$, where the Hamiltonian vector field $X^H=X^H_t$ is determined by the choice of a (time-dependent) Hamiltonian function $H=H_t$ and a linear symplectic form $\omega$ via $\omega(X^H_t,\cdot)=dH_t$. \\

Here a bilinear form $\omega:\IH\times\IH\to\IR$ on a real Hilbert space $\IH$ is called \emph{symplectic} if it is anti-symmetric and nondegenerate in the sense that the induced linear mapping $i_{\omega}:\IH\to\IH^*$ is an isomorphism. As in the finite-dimensional case it can be shown that for any symplectic form $\omega$ there exists a complex structure $J_0$ on $\IH$ such that $\omega$, $J_0$ and the inner product $\langle\cdot,\cdot\rangle$ on $\IH$ are related via $\langle\cdot,\cdot\rangle=\omega(\cdot,J_0\cdot)$.  \\

As an example consider the nonlinear Schr\"odinger equation on the circle $$i\del_t u \;+\; \Delta u \;-\; V(|u|^2,x,t)u\,=\, 0,$$ where $t\in\IR$ is the time coordinate, $x\in S^1=\IR/2\pi\IZ$ is the space coordinate with Laplace operator $\Delta=\del_x^2$, $u=u(t,x)\in\IC$ and $V$ is a smooth real-valued potential. Here it is easy to compute that it can be written as $$u_t\,=\, X^H_t(u),$$ where $X^H_t$ is the symplectic gradient of the (time-dependent) Hamiltonian function $$H_t(u)\,:=\, \frac{1}{2}\;\int_0^{2\pi} |\del_x u|^2 \;+\; f(|u|^2,x,t)\,dx$$ with respect to the symplectic form $\omega=\langle i\cdot,\cdot\rangle$ using the real inner product $\langle \cdot,\cdot\rangle$ on the Hilbert space $\IH=L^2(S^1,\IC)$ of square-integrable complex-valued functions on the circle, when $V$ is given by the derivative of $f$ with respect to the first coordinate.  \\

As the example of the nonlinear Schr\"odinger equation illustrates, it is a typical feature of Hamiltonian PDEs that the underlying Hamiltonian function is only defined on the dense subspace $\IH_0$ of $\IH$, $H=H_t:\IH_0\to\IR$. Since, in general, the Hamiltonian function will not be defined on the same function space where the bilinear from $\omega$ is symplectic, following \cite{K} one instead consider a nested sequence of separable Hilbert spaces $$\IH\supset\IH_0\supset\IH_1\supset\ldots\supset\IH_k\supset\ldots,$$ called a Hilbert scale, where one requires that each inclusion is dense and compact and we assume that $k\in\IN$ is related to the number of weak derivatives. \\

Since a Hamiltonian PDE can hence not be viewed as an ordinary differential equation on some Hilbert space, the Hamiltonian vector field $X^H_t$ typically does not possess a well-defined flow $\phi_t=\phi^H_t$ on the full symplectic Hilbert space $\IH$. While in general it is a very difficult task to prove the existence of such a flow, in this paper we want to focus on the case where $H_t:\IH\to\IR$ is a Hamiltonian function which is defined and smooth on the full Hilbert space $\IH$. Furthermore we want to assume that the Hamiltonian vector field is complete so that the flow is globally defined. \\

It is the goal of this paper to prove a non-squeezing result for infinite-dimensional Hamiltonian flows. Fix $R>0$ and denote by $\IC=\IR^2$ an arbitrary but fixed ($J_0$-)complex subspace of $\IH$ of real dimension $2$. Then $B^{\IH}(R)=\{x\in\IH: |x|< R\}$, $Z^{\IH}(R)=B^2(R)\times \IH/\IC$ denote the Hilbert ball and the Hilbert cylinder of radius $R>0$ (centered at $0\in\IH$), respectively. Using methods from non-standard analysis together with Gromov's original proof (\cite{Gr}) using pseudoholomorphic curves in finite dimensions, we prove the following   

\begin{theorem}\label{Main1}
For every smooth (time-dependent) Hamiltonian function $H:\IR\times \IH\to\IR$ on a separable symplectic Hilbert space there exists some positive time $T$ (characterized in terms of pseudoholomorphic curves) such that for all $t<T$ the flow map $\phi_t=\phi^H_t$ satisfies the natural analogue of Gromov's non-squeezing theorem:
If $\phi_t(v+B^{\IH}(r))\subset Z^{\IH}(R)$ for some $r>0$, $v\in\IH$, then $r\leq R$.
\end{theorem}

We emphasize that the time $T>0$ depends on the Hamiltonian $H$ as well as on the chosen complex subspace and on the radius of the zylinder. While it can be shown using an elementary proof that the Hamiltonian vector field $X^H_t$ of $H_t=H(t,\cdot)$ cannot be non-vanishing and pointing inward at every point on the boundary of the two-disk $B^2(R)\times\{0\}$, the latter is not sufficent to prove our statement. Most importantly, we actually give a geometrical characterization of the maximal time $T$ in terms of pseudoholomorphic spheres. \\

We start by outlining the natural infinite-dimensional generalization of the geometrical setup that Gromov used in his non-squeezing proof. First, note that for every $R>0$ and every $\epsilon>0$ there exists $\sigma>0$ such that there exist a symplectic embedding of $B^2(R)\subset\IC$ into the two-sphere $S^2=S^2(\sigma)$ of radius $\sigma$ with area $R^2\pi+\epsilon$. Let $J^{\IH}_t:=(\phi^{\IH}_t)_*J_0$ denote the family of compatible almost complex structures on $\IH$ obtained as pushforward of the complex structure $J_0$ on $\IH$ under the infinite-dimensional Hamiltonian flow $\phi^{\IH}_t=\phi_t$ on $\IH$. Using this homotopy together with a cut-off function to interpolate smoothly between $J^{\IH}_t$ on $\del B^2(R)$ and $J_0=J^{\IH}_0$, the push-forward compatible almost complex structure $J^{\IH}_t$ on $B^2(R)\times \IH/\IC\subset\IH$ can be extended to a compatible almost complex structure $\tilde{J}^{\IH}_t$ on $S^2(\sigma)\times\IH/\IC$. \\

In order to follow the strategy of Gromov's proof, the next step would be to prove, for every $t$, the existence of a $\tilde{J}^{\IH}_t$-holomorphic sphere $u=u_t$ in $S^2\times\IH/\IC$, that is, a map $u: S^2\to S^2\times\IH/\IC$ with $\CR_{\tilde{J}}(u)=Tu+\tilde{J}(u)\cdot Tu\cdot i=0$ for $\tilde{J}=\tilde{J}^{\IH}_t$, which furthermore passes through $p=p_t=\phi^{\IH}_t(0)\subset S^2\times\IH/\IC$ and is homotopic to $S^2\times\{0\}\subset S^2\times\IH/\IC$. Due to the lack of compactness in the infinite-dimensional case, we cannot expect that Gromov's result continues to hold in this form. In this paper, we are nonetheless able to prove the following  \\

\begin{theorem}\label{Main2} 
There exists $T>0$ and a $C^0$-continuous family $t\mapsto u_t$, $t\in [0,T)$ of smooth $\tilde{J}^{\IH}_t$-holomorphic spheres, where for $t=0$ the $\tilde{J}^{\IH}_0$-holomorphic curve is given by the inclusion $S^2\to S^2\times\{0\}\subset S^2\times \IH/\IC$. Furthermore, the maximal time $T$ is characterized by the following two alternatives: Either there exists a non-regular $\tilde{J}^{\IH}_T$-holomorphic sphere $u_{\infty}$ and a sequence $u_n=u_{t_n}$, $n\in\IN$ of $\tilde{J}^{\IH}_{t_n}$-holomorphic spheres with $t_n\to T$ as $n\to\infty$ which converge in the $C^0$-sense to $u_{\infty}$. Alternatively, there exists a sequence $u_n=u_{t_n}$, $n\in\IN$ of $\tilde{J}^{\IH}_{t_n}$-holomorphic spheres with $t_n\to T$ as $n\to\infty$ which has no $C^0$-convergent subsequence.
\end{theorem}

Here a $\tilde{J}$-holomorphic sphere $u$ is called non-regular if the linearization of the Cauchy-Riemann operator $\CR_{\tilde{J}}$ at $u$ has no bounded inverse. We remark that also in the finite-dimensional case one can only expect the analogous family of pseudoholomorphic spheres to exists for some maximal time which is now only characterized by the first alternative. Indeed, while it follows from Gromov compactness that there exists a compact one-dimensional moduli space when $t$ ranges between $0$ and $1$, the projection to the time interval $[0,1]$ has singular values, in general. On the other hand, as in finite dimensions every sequence of pseudoholomorphic spheres has a $C^0$-convergent subsequence, it follows that the second alternative does not apply. 

\begin{remark} We want to emphasize that, in order to prove the existence of the limiting non-regular pseudoholomorphic sphere, one already \emph{must exclude bubbling-off}, that is, use Gromov compactness. In essence our result proves that the lack of $C^0$-compactness is the only obstruction towards the generalization from finite to infinite dimensions. \end{remark}

Furthermore, we claim that with some extra work one could prove that the family is not just continuous but indeed smooth. \\

\noindent\emph{Non-squeezing in infinite dimensions}\\  

Before discussing the idea and details of the proof, let us first review some known results about the validity of non-squeezing in infinite dimensions: In the special case of semilinear Hamiltonian PDE where the flow is a smooth \emph{compact} perturbation of the linear flow of $A$, non-squeezing was already proven by Kuksin, see \cite{K}. This requires that the nonlinear part is not just smooth, but actually has a higher regularity, like e.g. a smoothing operator. Another result with regularity assumptions was established by Sukhov and Tumanov in \cite{ST}. Although most of the famous Hamiltonian PDE are not covered by the result in \cite{K}, non-squeezing was proven, for example, for the nonlinear Schr\"odinger equation and for the Korteweg-de Vries equation by Bourgain in \cite{Bo2} and by Colliander,  Keel, Staffilani, Takaoka and Tao in \cite{CKS+}, respectively. Both results rely on the fact that the corresponding infinite-dimensional flows can be uniformly approximated by finite-dimensional flows. Finally, in \cite{AM}, Abbondandolo and Majer proved non-squeezing for general infinite-dimensional symplectic flows under the assumption that the image of the unit ball stays convex, by defining an infinite-dimensional symplectic capacity. As main application they use their result to show that non-squeezing holds for symplectic flows for short times, under the assumption that there exist uniform bounds for the derivatives of $\phi_t$ on the Hilbert ball. Our result shows that the request for the uniform bounds in \cite{AM} can be dropped after restricting to the the case that the cylinder has a fixed radius and is centered at the origin. \\

\noindent\emph{A non-standard approach}\\

Apart from the fact that the non-squeezing problem in infinite dimensions has already received some attention over the recent years, the field of infinite-dimensional symplectic geometry is much less explored than its finite-dimensional counterpart. The main reason is that Gromov's theory of pseudoholomorphic curves does not generalize in an immediate way from finite-dimensional to infinite-dimensional symplectic manifolds. On the other hand, it is well-known that non-standard model theory, introduced by A. Robinson in his seminal book \cite{R}, provides a very efficient way to translate results from the finite to the infinite context, see e.g. the paper \cite{Os} on infinite-dimensional Brownian motion. We hope that this paper as well as our second paper \cite{F} on Floer homology will serve as a starting point of a general program to generalize results from finite-dimensional symplectic geometry to infinite dimensions using this tool. While a standard approach would require to prove appropriate infinite-dimensional generalizations of every technical result, starting from the existence of Banach space bundles over Banach manifolds all the way to bubbling-off and elliptic bootstrapping, our non-standard proofs only build on the well-established finite-dimensional results. In the case of non-squeezing, note that this includes the monotonicity theorem for minimal surfaces. \\

Indeed it is well-known, see e.g. \cite{L1}, \cite{L2} and \cite{Ke}, that there exist so-called \emph{non-standard} models of mathematics in which there exists an extension of the notion of finiteness: There exist new so-called unlimited *-real (and *-natural) numbers which are greater than all \emph{standard} real (and natural) numbers; in an analogous way there exist infinitesimal numbers, whose moduli are smaller than any positive standard real number. These *-real numbers can be introduced, using the axiom of choice, as equivalence classes of sequences of real numbers, where the standard numbers are included as constant sequences, while sequences converging to $\pm\infty$ or $0$ are examples of unlimited and infinitesimal numbers, respectively. In this paper we will use the resulting surprising fact that there exists a *-finite-dimensional symplectic space $\IF$, i.e., a finite-dimensional symplectic space in the sense of the non-standard model, which contains the infinite-dimensional Hilbert space $\IH$ as a subspace. Furthermore, on $\IH\subset\IF$, the smooth infinite-dimensional Hamiltonian flow $\phi_t=\phi^{\IH}_t$ can be represented by a *-finite-dimensional Hamiltonian flow $\phi^{\IF}_t$, that is, they agree up to an error which is smaller than any (standard) positive real number. While the existence of ideal elements such as $\IF$ (and of the flow $\phi^{\IF}_t$ on it) is established using the so-called \emph{saturation principle}, the existence of pseudoholomorphic curves in $\IH$ (more precisely, in $S^2\times\IH/\IC$) is then deduced from the classical finite-dimensional theory by employing the second fundamental principle of non-standard model theory, the so-called \emph{transfer principle}. This principles states that every statement that holds in finite dimensions and can be formulated in first-order logic has an analogue in the *-finite-dimensional setting. While this immediately allows us to deduce the existence of pseudoholomorphic curves in $\IF$ (more precisely, in $^*S^2\times\IF/{^*\IC}$, which contains $S^2\times\IH/\IC$ as a subset) for the almost complex structure $\tilde{J}^{\IF}_t$ defined using the *-finite-dimensional Hamiltonian flow $\phi^{\IF}_t$, we essentially use the minimal surface property of pseudoholomorphic curves and bubbling-off analysis (in finite dimensions) to prove that these non-standard $\tilde{J}^{\IF}$-holomorphic spheres actually give pseudoholomorphic spheres (in the standard meaning) for the almost complex structure $\tilde{J}^{\IH}$ on $S^2\times\IH/\IC$ defined using the infinite-dimensional flow $\phi^{\IH}_t$. \\ 

This paper is organized as follows: While in section two we first describe what we mean by the standard model of mathematics, using that all appearing mathematical entities are obtained by successively taking sets of sets starting from the real numbers, in section three we introduce all necessary concepts, definitions and theorems from non-standard model theory that are needed to follow the proof of our non-squeezing theorem. In section four we prove the existence of the *-finite-dimensional symplectic space $\IF$ containing $\IH$ and the fact that the infinite-dimensional flow $\phi^{\IH}_t$ on $\IH$ can be represented with all derivatives by the *-finite-dimensional flow $\phi^{\IF}_t$ which, together with the first two sections, are written in order to serve as foundation for future papers on infinite-dimensional Hamiltonian dynamics using non-standard methods. In section five we outline how non-squeezing in infinite dimensions can be proven using this *-finite-dimensional representation. While the monotonicity theorem holds directly by the transfer principle, the key ingredient is the proof of theorem \ref{Main2}, where an outline of the proof as well as the key lemmas are given in section six. In section seven and eight we finally give detailed proofs of these key lemmas and complete the proof in section eight. \emph{We would like to emphasize that this paper is written in such a way that it does not require any previous knowledge about non-standard model theory.}\\ 

The idea for this project came up after listing to Alberto Abbondandolo's talk on his joint work \cite{AM} at VU Amsterdam in December 2014.  Further I am deeply grateful to Horst Osswald from LMU Munich for introducing me to the fascinating world of non-standard analysis during my time as undergraduate student.

\section{The standard model}

In this section we provide an outline of all the background and relevant definitions and statements about nonstandard analysis that the reader needs to know in order to follow the rest of the paper. Here we describe the original model-theoretic approach of Robinson (\cite{R}), outlined in the excellent expositions \cite{L1}, \cite{L2} as well as in \cite{Ke}, to which we refer and which shall also be consulted for more details and background. \\

Believing in the axiom of choice it is well-known, see e.g. (\cite{L2}, theorem 2.9.10), that there exist non-standard models of mathematics in which, on one side, one can do the same mathematics as before (transfer principle) but, on the other side, all sets behave like compact sets (saturation principle). The idea is to successively introduce new ideal objects such as infinitely small and large numbers. The proof of existence of the resulting polysatured model is then performed in complete analogy to the proof of the statement that every field has an algebraic closure, by employing the axiom of choice. \\

A model of mathematics $V$ of a family of sets which is rich enough in order to do all the mathematics that one has in mind. Since for existence proof of non-standard models it is crucial that $V$ is still a set in the sense of set theory, there are (abstract) sets which are not in $V$. Below we show how to define such a set $V$ which contains all mathematical entities that we need for our proof. For most of the upcoming definitions and theorems on the general background on model theory we refer the reader to \cite{L2} as well as \cite{Ke}. The first definition is taken from the appendix in (\cite{L2}, section 2.9). 

\begin{definition} A sequence $V=(V_n)_{n\in\IN}$ of hierarchically ordered sets $V_n$, $n\in\IN$ is called a \emph{model} if the elements in $V_n$ are sets formed from the elements in $V_0,\ldots,V_{n-1}$, i.e.,  $V_n\subset\PP(V_0\cup\ldots\cup V_{n-1})$ and $V_0$, called the set of urelements, does not contains elements from higher sets, i.e., $V_0\cap\bigcup_{n\geq 1} V_n=\emptyset$. \end{definition}

By choosing the model $V=(V_n)_{n\in\IN}$ large enough, one can ensure that the models contains all mathematical entities that one wants to work with. Apart from assuming that every subset formed from elements in $V_0,\ldots,V_{n-1}$ is in $V_n$, below we show explicitly that for our proof it turns out to be sufficient to take the real numbers as urelements, i.e., $V_0=\IR$.

\begin{definition} We call $V=(V_n)_{n\in\IN}$ the \emph{standard model} if the urelements are the real numbers, $V_0=\IR$, and the model is full in the sense that $V_n=\PP(V_0\cup\ldots\cup V_{n-1})$. \end{definition}

In what follows, let $V=(V_n)_{n\in\IN}$ denote the standard model. As discussed in (\cite{L2}, 2.9), it follows that $$V(\IR)=\bigcup_{n=0}^{\infty} V_n(\IR)\;\;\textrm{with}\;\; V_n(\IR)=\IR\cup V_n\;\;\textrm{for all}\;\;n\in\IN$$ is the superstructure over the real numbers in the sense of (\cite{L2}, definition 2.1.1) and (\cite{Ke}, definition 15.4). Note that, for $n>1$, we have for every full model that $V_{n-1}\subset V_n$. \\

Since in analysis one considers sets of functions which themselves can be viewed as sets built from the real numbers, the superstructure over the real numbers contains all mathematical entities that one needs to do analysis, see (\cite{Ke}, section 15B). In particular we claim 

\begin{proposition} The standard model $V=(V_n)_{n\in\IN}$ contains (isomorphic copies of) all mathematical entities that we need in order to formulate and prove our non-squeezing theorem. 
\end{proposition}

\begin{proof} Instead of trying to give a proof listing all mathematical entities that will ever occur, we rather give the recipe and discuss the most important examples. This said, the proof of our proposition mostly relies on the following observation: \\

\emph{If $a$ and $b$ are sets in $V_n$, then every function $f:a\to b$ is an element of $V_{n+2}$ and every set of functions $f:a\to b$ is an element of $V_{n+3}$.} \\ 

For this it suffices to observe that, in set theory, a function $f:a\to b$ is identified with the subset $\{(x,f(x)): x\in a\}$ of $a\times b$ and for each $x\in a$, $y\in b$ the tuple $(x,y)$ is defined as the set $\{x,\{x,y\}\}$. \\ 

In order to see that all mathematical entities that we need are in the standard model, let us give the most relevant examples. 
\begin{itemize}
\item First, since $V_0=\IR$, it follows that $\IR$ as well as all its subsets like $\IN$, $\IZ$ are elements in $V_1$.
\item Since tuples $(r_1,\ldots,r_n)$ of real numbers can be written as sets of tuples $\{(r_1,1),\ldots,(r_n,n)\}$, we see that they are elements in $V_3$ and hence $\IR^{2n}\cong\IC^n$ belongs to $V_4$. The same holds true for the subsets $B^{2n}(r)$ and $Z^{2n}(R)$ and all other subsets.
\item By identifying each separable real Hilbert space $\IH$ with the space $\ell^2=\ell^2(\IR)$ of square-summable series of real numbers (using a choice of a countable complete basis) and using that every series of real numbers is given by a function $f:\IN\to\IR$, it follows that all elements in $\IH$ are elements in $V_3$ and $\IH$ as well as all its subsets like $\IH_1$ are elements in $V_4$. 
\item Every Hamiltonian function $H$ and its symplectic gradient $X^H$ as well as every linear complex structure belong to $V_6$. \item Since the dual spaces of the linear spaces living in $V_4$ consists of functions, they hence belong to $V_7$.
\item A linear symplectic form is an element in $V_9$. For this it does not matter whether we view it as a linear map from the linear space to its dual space ($\omega:\IH\to\IH^*$) or as a bilinear map on the linear space ($\omega:\IR^{2n}\times\IR^{2n}\to\IR$, using that $\IR^{2n}\times\IR^{2n}\in V_7$).
\item An almost complex structure $J$, i.e., a complex structure on the tangent bundle, is a function from the space to the set of linear complex structures. Since the latter is an element of $V_7$, it follows that $J$ is again an element of $V_9$.
\item Since the space of almost complex structures belongs to $V_{10}$, it follows that every one-parameter family $\tilde{J}: t\mapsto J_t$ of almost complex structures is an element of $V_{12}$. 
\item Every $J$-holomorphic map $u$ is a map between sets belonging to $V_4$ and hence belongs to $V_6$. Note that, although the defining almost complex structure only appears in $V_n$ for $n\geq 9$, the moduli space is a subset of the set of all maps and hence, by fullness of the model, is also an element in $V_7$. 
\end{itemize}\end{proof}

In order to show that Gromov's existence result of $J$-holomorphic curves indeed continues to hold in infinite dimensions, we will use that, by abstract model theory, his statement also holds in the non-standard model which we are going to discuss below. To make the underlying transfer principle precise, we quickly recall all the necessary background from first-order predicate logic that is needed. \\

The idea is that, just like all mathematical entities that we need are contained in the standard model $V=(V_n)_{n\in\IN}$, all statements that we will transfer can be formalized in first-order logic, that is, they are sentences in the language $\LL_V$ for our standard model $V$. In the same way as the details in the precise definition of models are not ultimatively important in order to understand the strategy of our proof, we continue to recall all needed foundations from logic for the sake of completeness of the exposition. For the following definitions we continue to refer to the appendix in (\cite{L2}, section 2.9) as well as (\cite{Ke}, section 15B).

\begin{definition} The \emph{alphabet} of the \emph{language} $\LL_V$ of the model $V=(V_n)_{n\in\IN}$ consists of the logical symbols $\vee$, $\neg$, $\exists$, $=$, $\in$, a countable number of variables, the elements in $V_{<\infty}:=\bigcup_{n\in\IN} V_n$ as parameters, and auxiliary symbols like parentheses. 
\end{definition}

\begin{definition} A \emph{sentence} in the \emph{language} $\LL_V$ of the model $V=(V_n)_{n\in\IN}$ is build inductively from the following rules: \begin{itemize}
\item[i)] If $a,b\in V_{<\infty}$, then $a\in b$ and $a=b$ are sentences in $\LL_V$.
\item[ii)] If $A$ and $B$ are sentences in $\LL_V$, then $A\vee B$ and $\neg A$ are sentences in $\LL_V$.
\item[iii)] Let $A$ be a sentence in $\LL_V$ and $a,b\in V_{<\infty}$ are parameters in $\LL_V$. If $x$ is a variable not occurring in $A$, then $\exists x\in a A_b(x)$ is a sentence in $\LL_V$, where $A_b(x)$ is obtained from $A$ by replacing each occurrence of the parameter $b$ in $A$ by the variable $x$. 
\end{itemize}
Every $A(x)=A_b(x)$ as in part iii) with a free variable $x$ is called a \emph{formula} in $\LL_V$. Furthermore, for every parameter $a\in V_{<\infty}$, by $A(x)(a)$ we denote the new sentence in $\LL_V$ obtained by replacing the variable $x$ by the parameter $a$. 
\end{definition}

Whether a sentence $A$ holds true in the model $V$, written $V\models A$, is decided using the usual interpretation for sentences in set theory, see (\cite{L2}, 2.9), (\cite{Ke}, 15B). 

\section{The non-standard model}

Using the axiom of choice one can prove that there exists a so-called \emph{non-standard model} in which the same mathematics hold true but in which every set from $V$ can be viewed as a precompact set. More precisely, after reformulating (\cite{L2}, theorem 2.9.10), we have the following 

\begin{theorem}\label{non-standard model}
Given the standard model $V=(V_n)_{n\in\IN}$ there exists a corresponding non-standard model $W=(W_n)_{n\in\IN}$, together with an embedding $\ast: V_{<\infty}\rightarrow W_{<\infty}$ respecting the filtration, i.e. $\ast_n: V_n\rightarrow W_n$, satisfying the following two important principles. \\
\begin{itemize}
\item\emph{Transfer principle:} If a sentence $A$ holds in the language $\LL_V$ of the model $V$, $V\models A$, then the corresponding sentence ${^*}A$, obtained by replacing the parameters from $V$ by their images in $W$ under $\ast$, holds in the language $\LL_W$ of the model $W$, $W\models {^*}A$. \\
\item\emph{Saturation principle:} If $(a_i)_{i\in I}$ is a collection of sets in $W$, indexed by a set $I$ in $V$, and satisfying $a_{i_1}\cap\ldots\cap a_{i_n}\neq\emptyset$ for all $i_1,\cdots,i_n\in b$, $n\in\IN$ \emph{(finite intersection property)}, then also the common intersection of all $a_i$, $i\in I$ is non-empty, $\bigcap_{i\in I} a_i\neq\emptyset$. \\
\end{itemize}
\end{theorem}

\begin{proof} Since in the references the theorem is not precisely stated in the above form, let us quickly describe how it can be deduced from \cite{L2}. In (\cite{L2}, theorem 2.9.10) it is claimed that there exists a so-called monomorphism from the superstructure $V(\IR)$ over $\IR$ into the superstructure $V({^*}\IR)$ over the set ${^*}\IR$ of non-standard real numbers. The latter are defined explicitly as equivalence classes of sequences of real numbers using the axiom of choice in (\cite{L2}, definition 1.2.3). Note that by (\cite{L2}, definition 2.4.3 and remark 2.4.4) the property of the map $\ast: V(\IR)\to V({^*}\IR)$ being a monomorphism is equivalent to the transfer principle, in particular, the latter indeed implies that $\ast$ respects the filtration. On the other hand, the fact that the formulation of the saturation principle given here is equivalent to the definition in (\cite{L2}, definition 2.9.1) is proven in (\cite{L2}, theorem 2.9.4), noticing that, by the definition of the cardinal number $\kappa^+$ appearing in (\cite{L2}, theorem 2.9.10), every set in $V$ is $\kappa^+$-small. Since the saturation property is only assumed when all sets $a_i$, $i\in V$ are \emph{internal} in the sense of (\cite{L2}, definition 2.8.1), that is, when they are elements in the $\ast$-image ${^*}V_n(\IR)\subset V_n({^*}\IR)$ of the set $V_n(\IR)\in V_{n+1}(\IR)$, we follow the strategy in the appendix of (\cite{L2}, section 2.9) and define the non-standard model $W=(W_n)_{n\in\IN}$ by setting $W_n:={^*}V_n$ for all $n\in\IN$. In particular, every set in the non-standard model $W=(W_n)_{n\in\IN}$ is internal.      
\end{proof}

In what follows we follow the usual conventions and write ${^*}a:=\ast(a)$ for every set $a\in V_{<\infty}\backslash V_0$ and identify $a:=\ast(a)$ for every urelement $a\in V_0=\IR$. \\
 
\begin{definition} A set $a$ is called
\begin{itemize}
\item[i)] \emph{internal} if $a\in W_{<\infty}$,
\item[ii)] \emph{standard} if $a={^*}b:=\ast(b)\in W_{<\infty}$ for some $b\in V_{<\infty}$.
\item[iii)]\emph{external} if $a$ is not internal.
\end{itemize} 
\end{definition}

We start with some immediate consequences of the transfer principle, see (\cite{L2}, proposition 2.4.6).

\begin{proposition} Let $a,b$ be sets in $V_{<\infty}$. Then we have
\begin{itemize}
\item[i)] $a=b$ if and only if ${^*}a={^*}b$,
\item[ii)] $a\in b$ if and only if ${^*}a\in {^*}b$,
\item[iii)] $a\subset b$ if and only if ${^*}a\subset {^*}b$,
\item[iv)] $f:a\to b$ if and only if ${^*}f: {^*}a\to {^*}b$.
\end{itemize}
\end{proposition}

These in turn lead to the following 
\begin{corollary}\label{extension} It follows  
\begin{itemize}
\item[i)] $\ast: V_{<\infty}\to W_{<\infty}$ is an embedding.
\item[ii)] For every set $b\in V_{<\infty}$ we have that  ${^*}[b]:=\{{^*}a: a\in b\}\subset {^*}b$. 
\item[iii)] For every function $f:a\to b$ we have that ${^*}f: {^*}a\to {^*}b$ is an extension of $f$ in the sense that for all $c\in a$ we have ${^*}(f(c))=({^*}f)({^*}c)\in {^*}b$.
\end{itemize}
\end{corollary}

\noindent\textbf{Examples:} 
\begin{itemize}
\item[i)] Since $+$ is a function from $\IR\times\IR$ to $\IR$, it follows that ${^*}+$ is a function from ${^*}\IR\times {^*}\IR$ to ${^*}\IR$ with ${^*}r {^*}+ {^*}s = {^*}(r+s)$ for all $a,b\in\IR$. 
\item[ii)] Since the symplectic form $\omega$ on $\IH$ is a map from $\IH\times\IH$ to $\IR$, its $\ast$-image $^*\omega$ is a map from ${^*\IH}\times{^*\IH}$ to $^*\IR$ which agrees with $\omega$ on $\IH\times\IH\subset{^*\IH}\times{^*\IH}$. Analogous statements hold true for the inner product $\<\cdot,\cdot\>$  and the complex structure $J_0$ on $\IH$. 
\item[iii)] Since, for all $n\in\IN$ and all $k\in\IN$, we know that $\sum_{i=1}^k$ is a function from $(\IR^n)^k$ to $\IR^n$, it follows that, now even for all $n\in {^*}\IN$ and all $k\in{^*\IN}$, ${^*}\sum_{i=1}^k$ is a function from $({^*}\IR^n)^k$ to ${^*}\IR^n$ with ${^*}\sum_{i=1}^k {^*}r_i={^*}\sum_{i=1}^{k-1} r_i + r_k$ for all $k\in{^*\IN}$. 
\item[iv)] Since every sequence $s=(s_n)_{n\in\IN}$ of natural numbers is a function from $\IN$ to $\IR$, it follows that its $\ast$-image ${^*}s$ is a function from ${^*}\IN$ to ${^*}\IR$ with ${^*}s_n= {^*}(s_n)$ for all $n\in\IN$. \\ 
\end{itemize}

We make the following \\

\noindent\textbf{Convention:} \emph{If no confusion is likely to arise, we make the convention to identify each standard set $b\in V_{<\infty}$ with ${^*}[b]\subset {^*}b\in W_{<\infty}$. In particular, we have $\IR\subset {^*}\IR$ and $\IN\subset {^*}\IN$.} \\

Indeed it is true that the saturation principle implies that the non-standard model $W=(W_n)_{n\in\IN}$ is (much) larger than the standard model $V=(V_n)_{n\in\IN}$. For the next statement we refer to (\cite{L2}, proposition 2.4.6) and (\cite{L2}, proposition 2.9.7).

\begin{proposition} We have the following dichotomy:
\begin{itemize}
\item[i)] If $b\in V_{<\infty}$ has finitely many elements, then its $\ast$-image ${^*}b\in W_{<\infty}$ consists of the $\ast$-images of its elements, $$^*\{a_1,\ldots,a_n\}=\{^*a_1,\ldots,^*a_n\}.$$
\item[ii)] If $b\in V_{<\infty}$ has infinitely many elements, then its $\ast$-image ${^*}b\in W_{<\infty}$ contains $b$ as a proper subset, $${^*}[b]=\{{^*}a: a\in b\}\subsetneq {^*}b.$$ 
\end{itemize}
In particular, it follows from ii) that $\ast: V_{<\infty}\to W_{<\infty}$ is a proper embedding.
\end{proposition}

\begin{proof}
While the part i) follows from the transfer principle after observing that the equality $b=\{a_1,\ldots,a_n\}$ can be incoded into the sentence $a\in b \Leftrightarrow a=a_1\vee\ldots\vee a=a_n$ in $\LL_V$, for part ii) consider the collection of sets $(a_i)_{i\in b}$ given by $a_i:={^*}b\backslash\{i\}$ for $i\in b$. While it easy to see that they have the finite intersection property, $a_{i_1}\cap\ldots\cap a_{i_n}\neq\emptyset$ for all $i_1,\cdots,i_n\in b$, $n\in\IN$, every element in $\bigcap_{i\in b} a_i\neq\emptyset$ is an element of ${^*}b\backslash b$. Note that, while in part i) the finite intersection property fails, in part ii) the transfer principle cannot be applied as the corresponding sentence would have infinite length, which is forbidden.   
\end{proof}

In particular, one can show that ${^*}\IR$, the set of *-real (or hyperreal or non-standard real) numbers, contains infinitesimals as well as numbers which are greater than any real number. 

\begin{proposition}\label{infinitesimal} The saturation principle implies the existence of the following ideal objects. 
\begin{itemize}
\item[i)] There exist $r\in{^*}\IR\backslash\{0\}$ such that $|r|<1/n$ for every standard natural number $n\in\IN$. Any such $r\in{^*}\IR$ (including $r=0$) is called \emph{infinitesimal} and we write $r\approx 0$.
\item[ii)] There exist $r\in{^*}\IR$ such that $|r|>n$ for every standard natural number $n\in\IN$. Any such $r\in{^*}\IR$ is called \emph{unlimited}. Any $r\in{^*}\IR$ which is not unlimited is called \emph{limited}. 
\item[iii)] A number $r\in{^*}\IR$ is limited if and only if it is \emph{near-standard} in the sense that there exists a standard real number $s\in\IR$ with $r-s\approx 0$. For every near-standard $r\in{^*}\IR$ we call ${^{\circ}}r:=s\in\IR$ the \emph{standard part} of $r$. 
\item[iv)] Any limited $n\in{^*}\IN$ is standard. 
\end{itemize}
\end{proposition}

\begin{proof} For the definitions we refer to (\cite{L2}, definitions 1.2.7 and 1.6.9). Since the existence of infinitesimal and unlimited numbers is the key reason why to care about non-standard analysis, let us give the short proof: Define for every $n\in\IN$ the sets $a_n:=\{r\in{^*}\IR: 0<|r|<1/n\}$ and $b_n:=\{r\in{^*}\IR: |r|>n\}$. Since the corresponding collections of sets obviously have the finite intersection property, we find that $\bigcap_{n\in\IN} a_n$ and $\bigcap_{n\in\IN} b_n$ are non-empty and any element in these sets has the desired properties. For the third part we refer to (\cite{L2}, proposition 1.6.11). Part iv) follows from the observation that there are only finitely many natural numbers smaller than a given one, so the $\ast$-image of the corresponding set does not contain any new elements. \end{proof}

\begin{remark} Along the same lines we have:
\begin{itemize}
\item[i)] Similar statements clearly hold when $^*\IR$ is replaced by $^*\IR^n$ for some standard $n\in\IN$. In particular, for every limited $r>0$ every point on ${^*S}^{n-1}(r)\subset{^*\IR}^n$ is near-standard, i.e., ${^*S}^{n-1}(r)$ is obtained from $S^{n-1}(r)$ by adding points which are infinitesimally close.
\item[ii)] In the same way as $^*\IR$ contains much more elements than $\IR$ itself, the non-standard extension $^*\IH$ of $\IH$ is a much larger space than $\IH$ itself.
\end{itemize}
\end{remark}   

In (\cite{L2}, theorems 1.6.8 and 1.6.15) it is shown that limited and infinitesimal numbers furthermore have the following nice closure properties. 

\begin{proposition}\label{limited-infinitesimal} We have
\begin{itemize}
\item[i)] Finite sums, differences and products of limited numbers are limited.
\item[ii)] Finite sums, differences and products of infinitesimal numbers are infinitesimal.
\item[iii)] The product of an infinitesimal number with a limited number is still infinitesimal.
\item[iv)] The standard part of a sum, difference or product of two limited numbers is the sum, difference or product of their standard parts.
\end{itemize}
\end{proposition}
 
\begin{remark} In an analogous way we will prove below that every infinite-dimensional (separable) Hilbert space $\IH$ is contained in a *-finite-dimensional Euclidean vector space $\IF$ of some unlimited but *-finite dimension $N\in{^*}\IN\backslash\IN$. The infinite-dimensional Hilbert space $\IH$ is \emph{not} a *-finite-dimensional Euclidean vector space itself, but is only \emph{contained} in some space which behaves as if it were finite-dimensional. \end{remark}

Apart from showing that the non-standard model contains infinitely-large numbers, the saturation principle immediately leads to the following, even more surprising fact, see (\cite{L2}, theorem 2.9.2).

\begin{proposition} For every standard set $b\in V_{<\infty}$ there exists a non-standard set $c\in W_{<\infty}$, which contains all elements of $a$, i.e., $a\in b$ implies ${^*}a\in c$, and which is \emph{*-finite} in the sense that there is a bijection from $c$ to an internal set $\{n\in{^*}\IN: n\leq N\}$ for some $N\in{^*}\IN$. \end{proposition}

Since a subset of a finite set in the standard model $V=(V_n)_{n\in\IN}$ is again a finite set, this seems to lead to an obvious logical contradiction. However, since the transfer principle only applies to subsets of *-finite sets which belong to the non-standard model, i.e., are internal themselves, the logical paradoxon is resolved in the following

\begin{proposition} For every infinite set $b\in V_{<\infty}$, the corresponding proper subset $b={^*}[b]=\{{^*}a: a\in b\}$ of ${^*}b$ is external. For example, $\IR$ and $\IN$ are external. In particular, the non-standard model is \emph{not full}, $W_n\subsetneq\PP(W_0\cup\ldots\cup W_{n-1})$. \end{proposition}
 
For the short proof we refer to (\cite{L2}, proposition 2.9.6). While these results are satisfactory from the theoretical point of view, for practical purposes it is rather important to know which subsets of an internal set in the non-standard model are still internal themselves, so that statements can be proven for them by applying the transfer principle. Since in applications one is almost exclusively interested in subsets which can be defined by requiring that their elements have a specific property, the following positive result originally due to Keisler, (\cite{Ke}, theorem 15.14), see also (\cite{L2}, theorem 2.8.4), is sufficient for all our purposes. 

\begin{proposition}\label{internal-definition} (Internal Definition Principle) Every \emph{definable set} belongs to the non-standard model $W=(W_n)_{n\in\IN}$ , that is, for every formula $A(x)$ in $\LL_W$ the set $\{a\in W_{<\infty}: W\models A(x)(a)\}$ is internal. \end{proposition}

In particular, every finite subset of an internal set is internal. \\

The fact that every internal set containing an infinite standard set as a subset must be strictly larger leads to the so-called \emph{spillover principles}, see (\cite{L2}, theorem 2.8.12). 

\begin{proposition}\label{spillover} Let $b$ denote an internal subset of ${^*}\IN$. Then it holds: 
\begin{itemize}
\item[i)] If for every $m\in\IN$ there exists some $n\geq m$ with $n\in b$, then $b$ must contain an unlimited *-natural number. 
\item[ii)] If for every unlimited *-natural number $N$ there exists a *-natural number $n\leq N$ with $n\in b$, then $b$ must also contain a standard natural number. 
\end{itemize}
\end{proposition}

\begin{proof} To prove i) define for every $m\in\IN$ the internal subset $b_m=\{n\in b: n\geq m\}$ of $b$. Since $b_{m_1}\cap\ldots\cap b_{m_k}\neq\emptyset$ for every finite collection $m_1,\ldots,m_k\in\IN$, it follows by the saturation principle that $\bigcap_{m\in\IN} b_m$ must contain an element which is an element of $b$ and greater than every standard natural number. In order to prove ii) observe that, by transfer, $b$ must have a minimal element $n$. If $n$ was unlimited, then there must exist $m\in b$ with $m\leq n-1$, contradicting the minimality of $n$. \end{proof}   

Finally, one of the main benefits of non-standard analysis is that the clumpsy $\epsilon$-formalism can be avoided by introducing infinitesimals and unlimited *-natural numbers. For the following proposition we refer to (\cite{L2}, theorem 1.7.1).

\begin{proposition}\label{convergence} 
A sequence $(s_n)_{n\in\IN}$ of real numbers converges to zero, $s_n\to 0$, as $n\to\infty$ if and only if $s_N:={^*s_N}\approx 0$ for all unlimited $N\in{^*}\IN\backslash\IN$. 
\end{proposition}

\begin{proof} First assume that $s_n\to 0$ as $n\to\infty$. By definition we know that for all $\eps>0$ there exists $n_0\in\IN$ such that $\forall n\in\IN: n\geq n_0\Rightarrow |s_n|<\eps$. By transfer, it follows that $\forall n\in{^*\IN}: n\geq n_0\Rightarrow |{^*s_n}|<\eps$. Since every unlimited $N\in{^*}\IN\backslash\IN$ is greater than every standard $n_0\in\IN$, it follows that $|{^*s_N}|<\eps$ for all standard $\eps>0$, that is, $|{^*s_N}|\approx 0$. In the opposite direction, assume that $|{^*s_N}|\approx 0$ for all unlimited $N\in{^*}\IN\backslash\IN$, in particular, for every standard $\eps>0$ and all unlimited $N$ there exists some $m\leq N$ such that $\forall n\in{^*\IN}: n\geq m\Rightarrow|{^*s_n}|<\eps$. By the spillover principle it follows that there must exist some standard $m\in\IN$ such that $\forall n\in\IN: n\geq m\Rightarrow|s_n|<\eps$, that is, $s_n\to 0$ as $n\to\infty$. \end{proof}

Since convergence in metric spaces is defined by requiring that the distance between points converges to zero, the above result immediately generalizes to all metric spaces.

\section{*-Finite-dimensional representations of infinite-dimensional flows}

As in the introduction, let $\IH=(\IH,\omega)$ be an infinite-dimensional separable symplectic Hilbert space and let $J_0$ denote the complex structure on $\IH$ which relates the symplectic form $\omega$ with the Hilbert space inner product via $\<\cdot,\cdot\>=\omega(\cdot,J_0\cdot)$. Furthermore let $H:\IR\times\IH\to\IR$, $H_t:=H(t,\cdot)$ be a smooth time-dependent Hamiltonian on the full Hilbert space $\IH$; in order to keep notation simple, we will ignore the $t$-dependence below. Just like the last two sections, also this section is of foundational nature and is intended to serve a reference for future work about the generalization of results in finite-dimensional Hamiltonian dynamics to the infinite-dimensional case using our non-standard approach. \\

The starting point of our proof is the observation that, using the saturation principle, every such Hamiltonian flow $\phi^{\IH}_t=\phi_t$ can be represented by a *-finite-dimensional Hamiltonian flow $\phi^{\IF}_t$. The underlying *-finite-dimensional space $\IF$ is characterized by the property that it is a complex subspace of ${^*\IH}$ which is large enough to contain the infinite-dimensional separable symplectic Hilbert space $\IH$ as a subspace. Recall that a complete orthonormal basis $(e_i)_{i\in\IN}$ for the symplectic Hilbert space $\IH$ is called unitary if, in addition, $J_0 e_{2i+1}=e_{2i+2}$ for all $i\in\IN\cup\{0\}$. \\

Let $\EE(\IH)$ denote the set of finite-dimensional ($J_0$-)complex subspaces of $\IH$. Then every element in the *-image ${^*\EE(\IH)}$ of $\EE(\IH)$ is called a *-finite-dimensional (${^*J_0}$-)complex subspace of ${^*\IH}$. Note that, since $\EE(\IH)$ is a set in the standard model $V$, it has a *-image in the non-standard model. Since every $F\in\EE(\IH)$ is a a subset of $\IH$ and a symplectic vector space over $\IR$ with respect to the restriction of $\omega$, every $F\in{^*\EE(\IH)}$ is a subset of  ${^*\IH}$ and a symplectic vector space over ${^*\IR}$ with respect to the restriction of the *-extension ${^*\omega}:{^*\IH}\times{^*\IH}\to{^*\IR}$ of $\omega$ by transfer. Note that ${^*\IR}$ is indeed a field by transfer, where addition and multiplication extend the corresponding operations on $\IR$ in the sense of corollary \ref{extension}. The real dimension defines a function $\dim:\EE(\IH)\to\IN$ in the standard model. By corollary \ref{extension} it follows that its *-image ${^*\dim}$ assigns to every *-finite-dimensional complex subspace $F$ a *-natural number $\dim(F):={^*\dim}(F)\in{^*\IN}$.

\begin{proposition}\label{existence}
There exists a *-finite-dimensional complex subspace $\IF$ of ${^*\IH}$ which contains the infinite-dimensional space $\IH$ as a complex subspace, $$\IH\subset\IF\subset{^*\IH}.$$ Furthermore, every complete unitary basis $(e_i)_{i\in\IN}$ can be extended to a unitary basis $(\tilde{e}_i)_{i=1}^{\dim\IF}$ of $\IF$ in the sense that $\tilde{e}_i=e_i$ for all $i\in\IN$, where the unlimited natural number $\dim\IF\in{^*\IN}\backslash\IN$ denotes the dimension of $\IF$ in the non-standard model, $\IF={^*\IR}^{\dim\IF}$.
\end{proposition}

\begin{proof} For a similar result we refer to (\cite{Os}, proposition 2.1). The proof of this surprising fact is, like the proof of the existence of infinitely large numbers, just an immediate consequence of the saturation principle for non-standard models: Concerning the existence of $\IF$, for every element $u\in\IH$ let $A_u$ denote the set of *-finite-dimensional complex subspaces $F$ of ${^*\IH}$ such that $u\in F$. Since for every finite collection $u_1,\ldots,u_n\in\IH$ there exists a finite-dimensional complex subspace containing $u_i$, $i=1,\ldots,n$, by saturation it follows that there must exist a *-finite-dimensional complex subspace, denoted by $\IF$, in the common intersection of all $A_u$, $u\in\IH$. For the extension of the complete unitary basis, we again employ the saturation principle: Since every finite collection $e_1,\ldots,e_n$ of basis vectors can be extended to a unitary basis $(\tilde{e}_i)_{i=1}^{\dim\IF}$ of $\IF$ by the transfer principle, it follows from saturation that there exists a unitary basis $(\tilde{e}_i)_{i=1}^{\dim\IF}$ of $\IF$ with $\tilde{e}_i=e_i$ for all $i\in\IN$.\end{proof}

For the rest of this exposition let us fix a complete unitary basis for the symplectic Hilbert space $\IH$ which we assume to be extended to a unitary basis of $\IF$, denoted by $(e_i)_{i=1}^{\dim\IF}$.  In particular, note that every $v\in\IF$ can be written as a *-finite sum, $$v=\sum_{i=1}^{\dim\IF} \<v,e_i\>\cdot e_i\;\in\IF.$$ \\

Recall from proposition \ref{infinitesimal} that every limited *-real number $r\in{^*\IR}$ is near-standard in the sense that there exists a standard real number $s\in\IR$ with $r\approx s$, called the standard part ${^{\circ}}r:=s$ of $r$. Replacing the inclusion $\IR\subset{^*\IR}$ of standard elements in a non-standard set by the inclusion $\IH\subset\IF$, we are lead to the following 

\begin{definition}\label{near-standard} An element $u\in\IF$ is called 
\begin{itemize}
\item[i)] \emph{limited} if its norm $|u|\in{^*\IR}$ is limited in the sense of proposition \ref{infinitesimal}, 
\item[ii)] \emph{near-standard} if there exists $v\in\IH$ with $u\approx v$ and we call $v$ the \emph{standard part} of $u$ and write ${^{\circ}u}:=v$. 
\end{itemize}
\end{definition} 

Here we say that $u\approx v$ for two elements $u,v\in\IF$ if their metric distance given by the *-Euclidean norm $|\cdot|=|\cdot|_{\IF}$ is infinitesimal, i.e., $|v-u|\approx 0$. \\

\begin{remark} Note that the *-extensions on $^*\IH$ of the symplectic form, the inner product and the complex structure on $\IH$ restrict to a symplectic form, inner product and complex structure on $\IF$ which extend the corresponding structures on $\IH\subset\IF$ by corollary \ref{extension}. In particular, we have $|v|_{\IF}\approx |{^{\circ}v}|_{\IH}$ for all near-standard $v\in\IF$ and we do not need to distinguish between the structures on $\IH$ and their *-finite-dimensional counterparts on $\IF$ in what follows. Furthermore, by transfer, the square of Euclidean norm $|\cdot|=|\cdot|_{\IF}$ on the *-finite-dimensional space $\IF$ is given by the *-finite sum $\sum_{i=1}^{\dim\IF}\<v,e_i\>^2$. \end{remark}
  
Of course, it is important to have a non-standard characterization of all near-standard points in $\IF$.  

\begin{proposition}\label{characterization}
A limited element $v\in\IF$ is near-standard if and only if for all unlimited $N\in{^*\IN}\backslash\IN$ we have $$\sum_{i=2N+1}^{\dim\IF} \<v,e_i\>^2\;\approx\; 0.$$ 
\end{proposition} 

\begin{proof} Assume that $v$ is near-standard, that is, there exists $u\in\IH\subset\IF$ with $v\approx u$. First, observe that by $|v|\approx |u|<\infty$, we know that $v$ is indeed limited. On the other hand, since $$\sum_{i=1}^{2n}\<u,e_i\>^2\to |u|^2(\in\IR^+)\,\,\textrm{as}\,\,n\to\infty$$ in the standard sense, it follows together with $|u|^2=\sum_{i=1}^{\dim\IF} \<u,e_i\>^2$ that for every standard $\eps>0$ there exists some standard $n\in\IN$ such that $$\sum_{i=2n+1}^{\dim\IF} \<u,e_i\>^2\,=\,\sum_{i=1}^{\dim\IF}\<u,e_i\>^2\;-\;\sum_{i=1}^{2n}\<u,e_i\>^2\,<\,\eps.$$ But from this it follows that for every unlimited $N\in{^*\IN}\backslash\IN$ we have that $\sum_{i=2N+1}^{\dim\IF} \<u,e_i\>^2<\eps$ for all standard $\eps>0$, that is, $$\sum_{i=2N+1}^{\dim\IF} \<v,e_i\>^2\,\approx\, \sum_{i=2N+1}^{\dim\IF} \<u,e_i\>^2\,\approx\,0.$$ 

In the opposite direction, we first observe that by the limitedness of $|v|^2=\sum_{i=1}^{\dim\IF}\<v,e_i\>^2$ we know that $\<v,e_i\>$ is limited and hence near-standard for all $i\in\IN$, and we want to show that $$v\,\approx\, u\,:=\,\sum_{i=1}^{\infty} {^{\circ}\<v,e_i\>} e_i\in\IH.$$ To this end fix again some standard $\eps>0$. Since $$\sum_{i=2N+1}^{\dim\IF} \<v,e_i\>^2\approx 0\,\,\textrm{and hence}\,\,\sum_{i=2N+1}^{\dim\IF} \<v,e_i\>^2<\eps$$ for all unlimited $N\in{^*\IN}\backslash\IN$, it follows from proposition \ref{spillover} that there must exist some standard $n\in\IN$ with $$\sum_{i=2n+1}^{\dim\IF} \<v,e_i\>^2<\eps,\,\,\textrm{in particular,}\,\,\sum_{i=2n+1}^{2m} {^{\circ}\<v,e_i\>^2}<\eps$$ for all standard $m\geq n$. Note that from this it follows that the limit $u=\sum_{i=1}^{\infty} {^{\circ}\<v,e_i\>} e_i\in\IH$ exists and it remains to be shown that $v\approx u$. But since for every standard $\eps>0$ there exists $n\in\IN$ with $$\Big|v-\sum_{i=1}^{2n}\<v,e_i\>e_i\Big|<\eps\,\,\textrm{and}\,\,\Big|u-\sum_{i=1}^{2n}\<u,e_i\>e_i\Big|<\eps,$$ it follows from $\sum_{i=1}^{2n}\<v,e_i\>e_i\approx\sum_{i=1}^{2n}\<u,e_i\>e_i$ that $|v-u|<2\eps$ for all standard $\eps>0$, that is, $v\approx u$. \end{proof} 

It still remains to introduce the *-finite-dimensional Hamiltonian flow $\phi^{\IF}_t$ on $\IF$ which represents the infinite-dimensional Hamiltonian flow $\phi^{\IH}_t=\phi_t$ on $\IH\subset\IF$. For this it however suffices to observe that the Hamiltonian function $H:\IR\times\IH\to\IR$ defining $\phi^{\IH}_t$ has a non-standard extension ${^*H}: {^*\IR}\times{^*\IH}\to{^*\IR}$. Apart from the fact that for every $t\in\IR$ we have that ${^*H}_t$ is an extension of $H_t$, as already mentioned above we will from now on suppress the time-dependence of the Hamiltonian for notational simplicity. \\

It follows from the transfer principle that the restriction of ${^*H}$ to the *-finite-dimensional subspace $\IF$ defines a *-finite-dimensional Hamiltonian flow $\phi^{\IF}_t$ on $\IF$. Note that, again by transfer, the corresponding Hamiltonian vector field $X^{\IF}$ is obtained by projecting to $\IF$ (along its symplectic complement) the *-image ${^*X^{\IH}}=X^{^*\IH}: {^*\IH}\to{^*\IH}$ of the Hamiltonian vector field $X^{\IH}$ defining $\phi^{\IH}_t$.  \\

\begin{remark}\label{differentiable} We quickly need to review the notion of differentiability in the non-standard sense: It follows from the transfer principle that the map $\phi^{\IF}:{^*\IR}\times\IF\to\IF$ is differentiable with respect to $t\in{^*\IR}$ in the non-standard sense and $\del_t \phi^{\IF}_t=X^{\IF}_t\circ\phi^{\IF}_t$. This means that for every $(t,u)\in{^*\IR}\times\IF$ and every $\eps\in{^*\IR}^+$ there exists $\delta\in{^*\IR}^+$ such that for all $t'\in{^*\IR}$ with $|t'-t|<\delta$ we have $$\frac{|\phi^{\IF}_{t'}(u)-\phi^{\IF}_t(u)-X^{\IF}_t(\phi^{\IF}_t(u))\cdot (t'-t)|}{|t'-t|}<\eps.$$ Note that the symplectic gradient $X^{\IF}$ is itself defined using the non-standard differential of the restriction of ${^*H}$ to $\IF$. \end{remark}

Using classical arguments of non-standard model theory, see \cite{Os}, it is now not very hard to show that, after employing the standard part map, the restriction of the *-finite-dimensional Hamiltonian flow $\phi^{\IF}_t$ to $\IH\subset\IF$ indeed agrees with the infinite-dimensional Hamiltonian flow $\phi^{\IH}_t$. 

\begin{proposition} For every near-standard $v\in\IF$ and every $t\in [0,1]$ we have that $\phi^{\IF}_t(v)\in\IF$ is near-standard and it holds $${^{\circ}\phi^{\IF}_t(v)}\;=\;\phi^{\IH}_t({^{\circ}v}).$$ \end{proposition}

\begin{proof} First it follows from proposition \ref{extension} that for every $u\in\IH\subset{^*\IH}$ we have $X^{^*\IH}(u)=X^{\IH}(u)$, which immediately implies that $X^{\IF}\approx X^{\IH}$ on $\IH\subset\IF$: For this observe that, since $$\sum_{i=1}^{2n} \<X^{\IH}(u),e_i\>^2 \leq  \sum_{i=1}^{\dim\IF} \<X^{^*\IH}(u),e_i\>^2 = |X^{\IF}(u)|^2$$ for all $n\in\IN$, it follows that $|X^{\IH}(u)|\leq ^{\circ}|X^{\IF}(u)|$, where $|\cdot|$ denotes the non-standard extension of the Hilbert space norm and $^{\circ}$ denotes the standard part map defined in proposition \ref{infinitesimal}. On the other hand, since $\IF$ is a subspace of ${^*\IH}$ and $X^{^*\IH}(u)=X^{\IH}(u)\in{^*\IH}$, we obtain $|X^{\IF}(u)|\leq |X^{\IH}(u)|$ and hence $$|X^{\IH}(u)-X^{\IF}(u)|\approx 0,\;\;\textrm{that is},\;\;X^{\IF}(u)\approx X^{\IH}(u).$$ On the other hand, since $X^{\IH}:\IH\to\IH$ is continuous at all $u\in\IH$, it follows from the non-standard characterization of convergence and continuity from proposition \ref{convergence} that $X^{^*\IH}(v)\approx X^{^*\IH}(u)=X^{\IH}(u)$ for all $v\in{^*\IH}$ with $v\approx u\in\IH$. Note that for this it is crucial that $u\in\IH$. Since $X^{\IF}$ is obtained by projecting $X^{^*\IH}$ to $\IF$, we further get for all $v\in\IF$ that $$v\approx u\in\IH\;\;\textrm{implies}\;\; X^{\IF}(v)\approx X^{\IF}(u)\approx X^{\IH}(u).$$ In other words, if $v\in\IF$ is near-standard, we proved that $X^{\IF}(v)\in\IF$ is near-standard with $^{\circ}X^{\IF}(v)=X^{\IH}({^{\circ}v}).$ But this in turn immediately implies that, if $v\approx u\in\IH$, then the resulting Hamiltonian flows satisfy $\phi^{\IF}_t(v)\approx\phi^{\IH}_t(u)$ for all $t\in\IR$. \end{proof}

\begin{remark} If we would know that $\phi^{\IH}_t$ can indeed be \emph{uniformly} approximated by finite-dimensional flows $\phi^n_t$, then by the same arguments we would get that $\phi^{\IF}_t\approx\phi^{^*\IH}_t$ does not only hold on $\IH\subset\IF$, but on all of $\IF$ (indeed it holds on all of $^*\IH$). The reason is that in our arguments above we would not need to fix the point $u\in\IH$. Note that in this case it is an easy exercise to show that the non-squeezing theorem for the *-finite-dimensional flow $\phi^{\IF}_t$, which simply holds by the transfer principle, immediately implies that non-squeezing holds for $\phi^{^*\IH}_t$ and hence, using the transfer principle backwards, also for the original infinite-dimensional flow $\phi^{\IH}_t$. On the other hand, it can indeed already be shown easily, see \cite{Bo2}, \cite{CKS+}, that the uniform approximatibility by finite-dimensional flows implies non-squeezing for the approximated infinite-dimensional flow. \end{remark}

Along the same lines, one can show that an analogous results holds true for all derivatives $\phi^{\IF}_t$ and $\phi^{\IH}_t$. Note that, by the transfer principle, for every standard $k\in\IN$ and every standard $u\in\IH$ the $k$.th derivative $T^k\phi^{^*\IH}_t(u)$ of the non-standard extension $\phi^{^*\IH}_t={^*\phi}^{\IH}_t$ of the Hamiltonian flow is given by the non-standard extension $^*(T^k\phi^{\IH}_t(u))$ of the $k$.th derivative $T^k\phi^{\IH}_t(u)$. Since no confusion is likely to arise, in what follows we will make no distinction between $T^k\phi^{\IH}_t(u)$ and its non-standard extension.  

\begin{proposition}\label{derivative} For every $k\in\IN$ and every near-standard $v\in\IF$ the $k$.th derivative $T^k\phi^{\IF}_t(v)$ of the *-finite-dimensional flow $\phi^{\IF}_t$ at $v\in\IF$ agrees, up to an infinitesimal error, with (the restriction to $\IF\subset {^*\IH}$ of the non-standard extension of) the $k$.th derivative $T^k\phi^{\IH}_t({^{\circ}v})$ of the infinite-dimensional flow $\phi^{\IH}_t$ at its standard part ${^{\circ}v}\in\IH\subset\IF$. In particular, for every limited $w^1,\ldots,w^k\in\IF$ it holds that $$T^k\phi^{\IF}_t(v)\cdot (w^1\otimes\ldots\otimes w^k)\;\approx\; T^k\phi^{\IH}_t({^{\circ}v})\cdot (w^1\otimes\ldots\otimes w^k),$$ and if $w^1,\ldots,w^k\in\IF$ are even near-standard, we have $$^{\circ}(T^k\phi^{\IF}_t(v)\cdot (w^1\otimes\ldots\otimes w^k)) = T^k\phi^{\IH}_t({^{\circ}v})\cdot({^{\circ}w^1}\otimes\ldots\otimes{^{\circ}w^k}).$$
\end{proposition} 

\begin{proof} The proof relies on the fact that, for every $k\in\IN$ and at every standard $u\in\IH\subset\IF$, the $k$.th derivative of $X^{\IF}$ agrees, up to an infinitesimal error, with the $k$.th derivative of the non-standard extension $X^{^*\IH}$ of the Hamiltonian vector field, i.e., $$T^k X^{\IF}(u)\cdot (w^1\otimes\ldots\otimes w^k)\;\approx\; T^k X^{^*\IH}(u)\cdot (w^1\otimes\ldots\otimes w^k)$$ for all limited $w^1,\ldots,w^k\in\IF$. While the case for $k=0$ was proven above, for the general case let us restrict to the case $k=1$ in order to keep the notation simpler. \\

First, in analogy to above, we use that $$\sum_{i=1}^{2n} \<T X^{\IH}(u),e_i\otimes e_j^*\>^2 \leq  \sum_{i=1}^{\dim\IF} \<TX^{^*\IH}(u),e_i\otimes e_j^*\>^2$$ for all $n\in\IN$ and all $1\leq j\leq \dim\IF$, where $(e_i^*)$ denotes the *-finite-dimensional dual basis to the orthonormal basis $(e_i)$ of $\IF$. On the other hand, since the $\IH^*\otimes\IH^*$-norm of $T^2 H(u)$ is finite, it follows directly that the limit  $$\sum_{i,j=1}^{\infty} \<TX^{\IH}(u),e_i\otimes e_j^*\>^2 =|TX^{\IH}(u)|^2<\infty$$ exists. Since the latter agrees with $|TX^{^*\IH}(u)|^2$ which is greater than or equal to $|TX^{\IF}(u)|^2$ given by $$|TX^{\IF}(u)|^2=\sum_{i,j=1}^{\dim\IF} \<TX^{^*\IH}(u),e_i\otimes e_j^*\>^2$$ as $X^{\IF}$ is obtained by projecting $X^{^*\IH}$ to $\IF\subset{^*\IH}$, we get that $$\sum_{j=1}^{\dim\IF}\<TX^{^*\IH}(u),e_j^*\> \cdot e_j^*\approx \sum_{j=1}^{\dim\IF}\<TX^{\IF}(u),e_j^*\>\cdot  e_j^*$$ and hence $TX^{^*\IH}(u)\cdot w\approx TX^{\IF}(u)\cdot w$ for all limited $v\in\IF$ using proposition \ref{limited-infinitesimal}. \\

In order to finish the proof, we again need to observe that the statement holds at near-standard points. Since $X^{\IH}$ is continuously differentiable, it follows from proposition \ref{convergence} that $$TX^{^*\IH}(v)\cdot w\approx TX^{^*\IH}(u)\cdot w\;\;\textrm{for all}\;\;v\approx u\in\IH\subset\IF.$$ On the other hand, since $TX^{\IF}$ is obtained by projecting $TX^{^*\IH}$ to $\IF\subset{^*\IH}$, we also have $TX^{\IF}(v)\cdot w\approx TX^{\IF}(u)\cdot w$ for all $v\approx u\in\IH\subset\IF$, which together gives $TX^{^*\IH}(v)\cdot w\approx TX^{\IF}(v)\cdot w$ for all near-standard $v\in\IF$ and limited $w\in\IF$. Since $$T\phi^{\IF}_0(v)\cdot w= w = T\phi^{^*\IH}_0(v)\cdot w$$ as well as $$\del_t T\phi^{^*\IH}_t(v)\cdot w=TX^{^*\IH}(v)\cdot w \;\;\textrm{and}\;\;\del_t T\phi^{\IF}_t(v)\cdot w=TX^{\IF}(v)\cdot w,$$ it follows that we indeed have $$T\phi^{^*\IH}_t(v)\cdot w\approx T\phi^{\IF}_t(v)\cdot w$$ for all near-standard $v\in\IF$ and limited $w\in\IF$. \\

The last statement in the proposition follows by simply observing that $T\phi^{^*\IH}_t(u)\cdot v = T\phi^{\IH}_t(u)\cdot v$ for all standard $u,v\in\IH\subset\IF$ by corollary \ref{extension}. \end{proof}

\section{Non-squeezing in infinite dimensions using non-standard models}

Instead of studying in how far Gromov's theory of pseudoholomorphic curves generalizes from finite to infinite dimensions, in this paper we show how non-standard model theory provides an alternative and effective way to generalize Gromov's result from finite to infinite dimensions. Here the crucial idea is that, using the transfer principle, it follows that all statements that Gromov used in his non-squeezing proof have analogues in the *-finite-dimensional setup. Apart from the monotonicity theorem for minimal surfaces, this in particular applies to Gromov's result about the existence of pseudoholomorphic spheres. While this immediately proves that non-squeezing holds true for the time-one map $\phi^{\IF}_1$ of the *-finite-dimensional flow, in this section we show that, in order to prove non-squeezing for the infinite-dimensional flow map $\phi^{\IH}_1$, it just remains to be shown that the corresponding $J$-holomorphic sphere used in Gromov's proof is near-standard. \\

In strict analogy to the construction above, let $J^{\IF}_t:=(\phi^{\IF}_t)_*J_0$ denote the family of compatible almost complex structures on $\IF$ obtained as pushforward of the complex structure $J_0$ on $\IF$ under the *-finite-dimensional Hamiltonian flow $\phi^{\IF}_t$ on $\IF$, where we again abbreviate $J^{\IF}=J^{\IF}_t$. Using this homotopy together with the same cut-off function as above to interpolate between $J^{\IF}$ on $\del B^2(R)$ and $J_0=J^{\IF}_0$, the push-forward compatible almost complex structure $J^{\IF}$ on ${^*B}^2(R)\times \IF/{^*\IC}\subset\IF$ can be extended to a compatible almost complex structure $\tilde{J}^{\IF}$ on ${^*S}^2(\sigma)\times\IF/{^*\IC}$, where we identify ${^*S}^2(\sigma)$ with ${^*S}^2$. The following proposition is an immediate consequence of Gromov's existence result for pseudoholomorphic spheres after applying the transfer principle. 

\begin{proposition}
There exists a *-differentiable family $t\mapsto u_t$ of $J^{\IF}_t$-holomorphic spheres, $t\in[0,T^{\IF}]$ where for $t=0$ the $J_0$-holomorphic curve is given by the inclusion ${^*S}^2\to {^*S}^2\times\{0\}\subset {^*S}^2\times \IF/{^*\IC}$. 
\end{proposition}

We remark that in the well-known standard finite-dimensional case one can indeed only expect the analogous family of pseudoholomorphic spheres to exists for some maximal $T>0$. Indeed, while it follows from Gromov compactness that there exists a compact one-dimensional moduli space if $t$ ranges from $0$ to $1$, the projection to the time interval $[0,1]$ has singular values, in general. By applying the transfer principle for non-standard models, we now in particular know that, up to some maximal time, there exists a $\tilde{J}^{\IF}_t$-holomorphic sphere $$u:{^*S}^2\to {^*S}^2\times \IF/^*\IC,\,\,Tu+\tilde{J}^{\IF}_t(u)\cdot Tu\cdot i=0$$ which is homotopic to the embedding ${^*S}^2\to {^*S}^2\times\{0\}\subset {^*S}^2\times \IF/^*\IC$. For the uniqueness of the $\tilde{J}^{\IF}_t$-holomorphic sphere we first make the convention that this holds true up to reparametrization and we always assume that it goes through the fixed point $p_t=\phi^{\IH}_t(0)$, where we use that, without loss of generality, $\phi^{\IH}_t(0)\in Z^{\IH}(R)$. The remainder of this paper is devoted to giving a proof of the following theorem which implies our second main theorem \ref{Main2}.  \\

\begin{theorem}\label{holo-spheres} There exists some standard $T=T^{\IH}\in(0,1]$ such that for every $0\leq t<T$ the unique $\tilde{J}^{\IF}_t$-holomorphic sphere $u_t: {^*S}^2\to {^*S}^2\times \IF/{^*\IC}$ is near-standard in the sense that for every $z\in{^*S}^2$ the image point $u(z)$ is near-standard. Moreover, after applying the standard part map, one obtains a smooth $\tilde{J}^{\IH}_t$-holomorphic sphere ${^{\circ}u}_t: S^2\to S^2\times \IH/\IC$ in the standard sense and the maximal time $T^{\IH}$ can be characterized as in theorem \ref{Main2}. \end{theorem}

Note that here near-standardness is defined with respect to the metric on ${^*S}^2$ inherited from the ambient Euclidean space. Using the fact, see the remark after proposition \ref{infinitesimal}, that every point on ${^*S}^2(\sigma)$ is near-standard, the characterization of near-standardness for points in $\IF$ from proposition \ref{characterization} generalizes naturally from $\IF={^*\IR}^{\dim\IF}$ to ${^*S}^2\times \IF/{^*\IC}$ after identifying $\IF/{^*\IC}$ with ${^*\IR}^{\dim\IF-2}$. We emphasize that for the proof of non-squeezing theorem \ref{Main1} we only need the near-standardness and do \emph{not} need the second half of the statement of theorem \ref{holo-spheres} about the (standard) smoothness of ${^{\circ}u}_t$ as we finish the proof of theorem \ref{Main1} in the non-standard setup.   \\

Before we turn to the proof of this theorem, let us show how this result is used to prove our main theorem \ref{Main1} on non-squeezing. As we continue to work in the *-finite-dimensional set-up, note our proof only needs the finite-dimensional version of the monotonicity theorem for minimal surfaces. \\

As in Gromov's original proof it follows from the definition of $\tilde{J}^{\IF}$ that $\Sigma:=(\phi^{\IF})^{-1}(u(^*S^2)\cap Z^{\IF}(R))$ is a minimal surface for the Euclidean metric on $\IF$. Knowing that $^*u(S^2)$ sits in an infinitesimal neighborhood of $S^2(\sigma)\times \IH/\IC$ in $^*S^2(\sigma)\times \IF/^*\IC$ by theorem \ref{holo-spheres}, that is, every point in $\phi^{\IF}(\Sigma)$ is infinitesimally close to $Z^{\IH}(R)\subset Z^{\IF}(R)$, it follows from $\phi^{\IH}(B^{\IH}(r))\subset Z^{\IH}(R)$ and $\phi^{\IF}\approx \phi^{\IH}$ on $B^{\IH}(r)\subset B^{\IF}(r)$ that, up to an infinitesimal error, the area of $\Sigma\cap B^{\IF}(r)$ must be greater or equal to $r^2\pi$. Note that here we further use that, up to an infinitesimal error, $\Sigma$ passes through the origin $0\in B^{\IH}(r)\subset B^{\IF}(r)$ and that the area estimate for minimal surfaces through the origin of $B^{\IF}(r)$ continues to hold by the transfer principle. On the other hand, since the area of $u(^*S^2)$ is equal to the area of $^*S^2(\sigma)$, it follows that $r^2\pi\leq R^2\pi + \epsilon$, which still implies $r\leq R$ as $\epsilon>0$ can be chosen arbitrarily small.

\begin{remark} While non-squeezing for $\IR^{2n}$ implies non-squeezing for $\IR^{2m}$ for all $m\leq n$ as $B^{2n}(r)\subset B^{2m}(r)\times\IR^{2n-2m}$ and $Z^{2n}(R)=Z^{2m}(R)\times\IR^{2n-2m}$, in contrast we emphasize that non-squeezing for $\IF$ (in the non-standard sense) does \emph{not} automatically imply non-squeezing for $\IH$. In essence, this follows from the fact that, since $\IH$ is not an internal set in the non-standard model, there is no complementary (internal) subspace of $\IH$ in $\IF$.\end{remark}

In order to prove theorem \ref{holo-spheres}, the crucial ingredient is the relation between the family of non-standard almost complex structures $\tilde{J}^{\IF}$ on ${^*S}^2\times\IF/{^*\IC}$ and the almost complex structure $\tilde{J}^{\IH}$ on $S^2\times\IH/\IC$. As both are directly defined using the first derivatives $T\phi^{\IF}_t$ and $T\phi^{\IH}_t$ of the flows $\phi^{\IF}$ and $\phi^{\IH}$ on $\IF$ and $\IH$, respectively, the following lemma is an immediate corollary to proposition \ref{derivative}. 

\begin{lemma}\label{almost-complex} For every near-standard $v\in{^*S}^2\times\IF/{^*\IC}$ we have that $\tilde{J}^{\IF}(v)$ agrees with (the restriction of the non-standard extension of) $\tilde{J}^{\IH}(^{\circ}v)$ up to an infinitesimal error; an analogous result holds true for the $k$.th derivatives of $\tilde{J}^{\IF}$ and $\tilde{J}^{\IH}$ at near-standard points $v\in\IF$ for all $k\in\IN$. In particular, for all limited $w\in T_v({^*S}^2\times\IF/{^*\IC})\cong\IF$ we have $$\tilde{J}^{\IF}(v)\cdot w \approx {^*\tilde{J}}^{\IH}(^{\circ}v)\cdot w,$$ and for every near-standard $w\in T_v({^*S}^2(\sigma)\times\IF/{^*\IC})\cong\IF$ we have that $\tilde{J}^{\IF}(v)\cdot w\in T_v({^*S}^2(\sigma)\times\IF/{^*\IC})\cong\IF$ is near-standard with $$^{\circ}(\tilde{J}^{\IF}(v)\cdot w)\;=\; \tilde{J}^{\IH}({^{\circ}v})\cdot {^{\circ}w}.$$ \end{lemma}

The proof of theorem \ref{holo-spheres} relies on the following observations: \\

First, since the compatible almost complex structure $\tilde{J}^{\IF}_0$ agrees with the standard product complex structure $J_0$ on ${^*S}^2\times \IF/{^*\IC}$, we know that for $t=0$ the unique $J_0$-holomorphic curve is given by the inclusion ${^*S}^2\to {^*S}^2\times\{0\}\subset {^*S}^2\times \IF/{^*\IC}$ and hence is infinitesimally close to $S^2\times\{0\}\subset S^2\times \IH/\IC$. In order to ensure that near-standardness is preserved as $t$ is increased, the crucial ingredient is to show that the derivative $\xi=\xi_t=\del u_t/\del t$ of the path is near-standard as long as the underlying holomorphic sphere $u_t$ is already known to be near-standard. Since $\xi=\xi_t$ solves a linear Cauchy-Riemann-type equation $\CR\xi+S_t\xi=\eta_t$, we in particular need to ensure that the linearized Cauchy-Riemann operator $D_t=\CR+S_t$ has an inverse with limited norm, so that, in particular, the solution is unique. While we show that this holds true up to some maximal time $T$, we emphasize that the near-standardness of $u=u_t$ is also needed to show that the first derivatives of $u$ are limited numbers, that is, not unlimited in the above sense. We exclude unlimitedness of the first derivatives by slightly modifying Gromov's bubbling-off argument, building on the finiteness of energy (in the standard sense) and that the mean-value inequality holds with non-infinitesimal constants. After establishing the limitedness of the first derivative, we use elliptic regularity (building again on the near-standardness of $u$) and the Sobolev embedding theorem to show that we actually obtain a smooth $\tilde{J}^{\IH}$-holomorphic sphere in the standard sense, where we further explore the relation between derivatives in the standard and in the non-standard sense. We emphasize that the existence of the underlying Banach manifold of maps $\BB^{1,p}=H^{1,p}({^*S}^2,{^*S}^2\times \IF/{^*\IC})$ is ensured, as usual, by the transfer principle and the well-known result in finite dimensions. However, we emphasize that now weak derivatives and the finiteness of their Lebesgue norms are clearly to be understood in the non-standard sense. \emph{While we will only need to employ that every result (like the Sobolev embedding theorem) has an analogue in the non-standard model by transfer, we need to deal with the fact that all arising norms will a priori be unlimited in the sense of proposition \ref{infinitesimal}.} \\

Summarizing, the proof of the theorem relies on the following two lemmas. 

\begin{lemma}\label{no-bubbling} Assume that the $\tilde{J}^{\IF}_t$-holomorphic sphere $u_t: {^*S}^2\to {^*S}^2\times \IF/{^*\IC}$ is near-standard. Then $u_t$ has limited first derivatives, i.e., the supremum norm $\|Tu_t\|_{\infty}$ (induced by the Riemannian metric $g^{\IF}_t=\omega(\cdot,\tilde{J}^{\IF}_t\cdot)$ on ${^*S}^2\times \IF/{^*\IC}$) is a limited *-real number.  After applying the standard part map, the resulting map ${^{\circ}u}_t: S^2\to S^2\times \IH/\IC$ is indeed smooth in the standard sense and satisfies $$T{^{\circ}u}_t+\tilde{J}^{\IH}_t(^{\circ}u)\cdot T{^{\circ}u}_t\cdot i=0.$$ \end{lemma}

\begin{lemma}\label{stay-nearstandard} There exists some standard $T>0$ with the following property: For $0\leq t<T$ assume that the $\tilde{J}^{\IF}_t$-holomorphic sphere $u_t: {^*S}^2\to {^*S}^2\times \IF/{^*\IC}$ is near-standard. Then the derivative $\xi_t=\del u_t/\del t$ of the path $t\mapsto u_t$ is near-standard. \end{lemma}

The near-standardness of all the tangent vectors $\xi_t=\del u_t/\del t$, together with a Lipschitz-type estimate which we will give in the last section, indeed shows that the $\tilde{J}^{\IF}_t$-holomorphic spheres in ${^*S}^2\times \IF/{^*\IC}$ stay infinitesimally close to $S^2\times\IH/\IC$ as $t$ runs from $0$ to $T$. Finally we show that the nearstandardness of $\xi_t$ implies that the path $t\mapsto u_t$ is indeed $C^0$-continuous and we prove the characterization of the maximal time $T$ using pseudoholomorphic spheres as outlined in theorem \ref{Main2}.  

\section{Nearstandard $J$-holomorphic spheres are smooth}

In this section we give a proof of lemma \ref{no-bubbling}. Since here the $t$-parameter is fixed, we will from now on drop the subindex $t$ for notational simplicity.  For the proof we use a non-standard version of the classical bubbling-off argument from (\cite{MDSa}, section 4.2) together with elliptic regularity from (\cite{MDSa}, section B.4). We emphasize that it will turn out to be crucial that the $\tilde{J}^{\IF}$-holomorphic sphere $u$ is near-standard in order to ensure that all appearing constants are limited and non-infinitesimal using corollary \ref{almost-complex}.  Furthermore we use here that the almost complex structure $\tilde{J}^{\IH}$ and the induced Riemannian metric $g^{\IH}$ on $S^2\times\IH/\IC$ are smooth in the standard sense. \\

\noindent\emph{Step 1: Limitedness of $Tu$ using bubbling-off}\\  

Note that, while it follows from the transfer principle that the $\tilde{J}^{\IF}$-holomorphic sphere $u$ is smooth, in particular, the supremum norm of its first derivative is finite, this statement clearly only holds in the non-standard sense. In particular, we in general need to expect that this supremum norm is an unlimited *-real number. In this section we first show that the supremum norm of $Tu$ indeed is limited. The crucial ingredient is that we know a priori that the energy $E(u)$ of $u$ is limited, as it equals $R^2\pi+\epsilon$. We then prove the limitedness of the first derivative by contraction, replacing in the usual bubbling-off proof convergence to infinity by unlimitedness. On the other hand, in order to be able to use the mean-value inequality (\cite{MDSa}, lemma 4.3.1) in the main step, we need to ensure that the derivatives of the induced Riemannian metric $g^{\IF}={^*\omega}(\cdot,\tilde{J}^{\IF}\cdot)$ are limited and/or not infinitesimal, respectively. \\

Assume that $\|Tu\|_{\infty}=\max\{|Tu(z)|:z\in{^*S}^2\}=C$ is an unlimited *-real number and choose $z_0\in{^*S}^2$ such that $|Tu(z_0)|=C$, where we assume without loss of generality that $z_0=0$. As in the classical bubbling-off proof we define $v:{^*B}^2(\sqrt{C})\to{^*S}^2\times\IF/{^*\IC}$ by $v(z):=u(z/C)$, such that $|Tv(0)|=1$ and $|Tv(z)|\leq 1$ for all $z\in{^*B}^2(\sqrt{C})$.  Because $C\in{^*\IR}^+$ was assumed to be unlimited, note that ${^*B}^2(\sqrt{C})\subset{^*\IC}$ is a disk of unlimited radius; in particular, it contains the full complex plane $\IC$ as a subset. For each $r\in[0,\sqrt{C}]\subset{^*\IR}^+$ define $\gamma_r: {^*S}^1\to {^*S}^2\times\IF/{^*\IC}$ by $\gamma_r(t):=v(re^{2\pi t})$. \\

Assume that there exists some standard $R\in\IR^+\subset{^*}\IR^+$ such that the length $\ell(\gamma_r)$ (measured again using the Riemannian metric $g^{\IF}$ induced by $\tilde{J}^{\IF}$) satisfies $\ell(\gamma_r)\geq\delta$ for all $r\geq R$ for some \emph{non}-infinitesimal $\delta>0$. Then, analogous to Gromov's original proof of bubbling-off replacing infinite by unlimited, the energy $E(v)=\int v^*\omega$ of $v$ is unlimited by conformality. Hence we know that for all non-infinitesimal $\delta>0$ and all standard $R>0$ there exists some $r\geq R$ such that $\ell(\gamma_r)<\delta$. But now, using proposition \ref{spillover}, there must exist some unbounded $0<R\leq \sqrt{C}$ such that the length of the corresponding loop is infinitesimal, $\ell(\gamma_R)\approx 0$. On the other hand, since the pullback of the symplectic form under the restriction $v_R$ of $v$ to ${^*B}^2(R)\subset {^*B}^2(\sqrt{C})$ is exact, it follows from Stokes' theorem for the energy of $v_R$ that $E(v_R)=\int v_R^*\omega \approx 0$. \\

In order to get the required contradiction, it remains to use the mean-value inequality as in Gromov's original proof. As mentioned above, this is point where it becomes important that the image of $u$ is near-standard, since this implies, by corollary \ref{almost-complex}, that on the image of $u$ the derivatives of $g^{\IF}$ agree with the derivatives of $g^{\IH}=\omega(\cdot,\tilde{J}^{\IH}\cdot)$ up to infinitesimal error. Since $g^{\IH}$ is smooth in the standard sense, we get that the derivatives of $g^{\IF}$ on the image of $u$ have limited supremum norm, which in turn proves that the local energy bound in (\cite{MDSa}, lemma 4.3.1) is non-infinitesimal. More precisely, we get that there exists some non-infinitesimal $\delta>0$ such that $$\int_{^*B^2(R)} |Tv|^2 <\delta\;\;\Rightarrow\;\; |Tv(0)|\leq \frac{8}{r^2}\cdot \int_{^*B^2(R)}|Tv|^2.$$ In particular, from $E(v_R)\approx 0$, it hence immediately follows that $|Tv(0)|\approx 0$, in contradiction to $|Tv(0)|=1$. \\

\noindent\emph{Step 2: Smoothness of $^{\circ}u$ using elliptic regularity}\\

After applying the transfer principle to the elliptic estimates used to prove regularity of pseudoholomorphic curves in the finite-dimensional setting, we can even prove that ${^{\circ}u}$ is smooth in the standard sense. For this we again will crucially use that $u$ is near-standard so that corollary \ref{almost-complex} applies, together with the smoothness of $\tilde{J}^{\IH}$ in the standard sense. \\

To this end, fix $z\in{^*S^2}$. Denoting by $\varphi_z:{^*\IC}\to{^*S^2}$ the canonical coordinate chart mapping $0\in{^*\IC}$ to $z\in{^*S^2}$, let us define $u_z:=\diag(\varphi_{u(z)}^{-1},\Id)\circ u\circ\varphi_z$. Fix a direction $\theta\in S^1$ and define the map $f=f_{z,\theta}:{^*[0,r)}\to\IF$ by $f(x)=u_z(x\cdot e^{i\theta})$. Since $u$ is near-standard and *-smooth, the same holds true for $f$, furthermore the limitedness of the $C^{\ell}$-norm of $u$ obviously implies the corresponding result for $f$. \\

As the first step we prove that ${^{\circ}f}$ is Lipschitz continuous. After applying the transfer principle to the intermediate value theorem and using that $f$ is differentiable in the non-standard sense, note that we have for every $x<y\in{^*[0,r)}$ that $$\frac{f(y)-f(x)}{y-x}=f'(w)\,\,\textrm{for some}\,\,w\in{^*[x,y]},$$ which implies that $|f(y)-f(x)|_{\IF}\leq{^*\|f\|_{C^1}}\cdot |y-x|$. By applying the standard part map, it follows that ${^{\circ}f}$ is Lipschitz continuous, $$|{^{\circ}f}(y)-{^{\circ}f}(x)|_{\IH}\leq c_1\cdot |y-x|,$$ where the positive real number $c_1\geq 0$ is the standard part of the limited number ${^*\|f\|_{C^1}}\in{^*\IR}^+\cup\{0\}$. \\

Since the supremum norm of $Tu$ as well as the area of ${^*S}^2$ are limited, it follows by transfer that the $(1,p)$-norm of $u$ is a limited number. Note that here and below all norms are understood in the non-standard sense; the only thing that we need in our proof is that every result about them which hold in the standard finite-dimensional case also hold in the non-standard model by the transfer principle. As in the proof of (\cite{MDSa}, theorem B.4.1) get from the elliptic estimates that also the $(k,p)$-norm of $u$ is limited for every standard natural number $k$. For the induction step observe that the composition $\tilde{J}^{\IF}\circ u$ has a limited $(k-1,p)$-norm if this holds for $u$ and for every standard $\ell$ the $\ell$.th derivative of $\tilde{J}^{\IF}$ has a limited supremum norm on the image of $u$. In order to ensure the latter, it simply suffices to use corollary \ref{almost-complex}, together with the fact the supremum norm of the $\ell$.th derivative of $\tilde{J}^{\IH}$ on the image of $u$ is finite in the standard and hence limited in the non-standard sense. Here we use that we already know that $^{\circ}u$ is continuous and hence the image of $^{\circ}u$ is compact (in the standard sense). \\

Finally, after transferring the Sobolev embedding theorem, we get that the supremum norm of the $\ell$.th derivative of $u$ is limited for every standard natural number $\ell$. In the same way as the limitedness of the $C^1$-norm could be used to prove the continuity of ${^{\circ}u}$, we can show that the limitedness of the $C^{\ell}$-norms in the non-standard sense for all standard $\ell\in\IN$ can be used to prove that $^{\circ}u$ is smooth in the standard sense. \\

Since we have already shown that ${^{\circ}f}$ is continuous in the standard sense, as the next step we now prove that it is differentiable. For this we observe that, by the same arguments as used for $f$, its first derivative $f':[0,r)\to\IF$ is Lipschitz with limited Lipschitz constant given by ${^*\|f\|_{C^2}}$. For the difference quotient used to establish the Lipschitz continuity, this can be used to prove that $$\Big|\frac{f(y)-f(x)}{y-x}-f'(x)\Big|_{\IF}=|f'(w)-f'(x)|_{\IF}\leq {^*\|f\|_{C^2}}\cdot |y-x|$$ using that $w\in{^*[x,y]}$. It follows that for every standard $\eps>0$ there exists a standard $\delta=\eps/c_2>0$ with $c_2=\max\{{^{\circ}({^*\|f\|_{C^2}})},1\}$ with the property that $$|y-x|<\delta\,\,\textrm{implies}\,\,\Big|\frac{f(y)-f(x)}{y-x}-f'(x)\Big|_{\IF}<\eps.$$ Now using that the near-standardness of $f$ implies that the above difference quotient is near-standard with standard part given by $${^{\circ}\Big(\frac{f(y)-f(x)}{y-x}\Big)}=\frac{{^{\circ}f}(y)-{^{\circ}f}(x)}{y-x},$$ it follows that $$|y-x|<\delta\,\,\textrm{implies}\,\,\Big|\frac{{^{\circ}f}(y)-{^{\circ}f}(x)}{y-x}-f'(x)\Big|_{^*\IH}<\eps.$$ But this proves that ${^{\circ}f}$ is differentiable at $x\in[0,r)$ in the standard sense with derivative given by the standard part of $f'(x)$; in particular, we see a posteriori that $f'$ has to be near-standard itself. \\  

On the other hand, after replacing $f$ by $f'$ and employing the limitedness of ${^*\|f'\|_{C^2}}\leq{^*\|f\|_{C^3}}$, one can successively prove that ${^{\circ}f}$ is infinitely often differentiable, that is, smooth in the standard sense.  Furthermore observe that the latter also proves that the standard derivatives are obtained from the non-standard derivatives by simply taking their standard part. \\

Finally, in order to see that the $\tilde{J}^{\IF}$-holomorphicity of $u$ implies that ${^{\circ}u}: S^2\to S^2\times \IH/\IC$ satisfies $$T{^{\circ}u}+\tilde{J}^{\IH}_t(^{\circ}u)\cdot T{^{\circ}u}\cdot i=0,$$ it just suffices to use corollary \ref{almost-complex}, i.e., the fact that on $S^2\times \IH/\IC\subset {^*S}^2\times \IF/{^*\IC}$ the *-finite-dimensional almost complex structure $\tilde{J}^{\IF}$ agrees with the infinite-dimensional almost complex structure $\tilde{J}^{\IH}$. Indeed the latter holds up to an infinitesimal error, but this however becomes invisible after taking the standard part map. 

\begin{remark} Let us emphasize and clarify the relevance of near-standardness of $u$ in our arguments. In order to get limitedness of the appearing derivatives of $\tilde{J}^{\IF}$ ($g^{\IF}$) we use that, by corollary \ref{almost-complex}, at \emph{near-standard} $u$, the norm of the $\ell$.th derivative of $\tilde{J}^{\IF}$  agrees with the norm of the non-standard extension of the $\ell$.th derivative of $\tilde{J}^{\IH}$ at its standard part $^{\circ}u$ for every $\ell\in\IN$. Since the latter agrees, by the transfer principle, with the norm of the $\ell$.th derivative of $\tilde{J}^{\IH}$ at $^{\circ}u$ in the standard sense, we can use the smoothness of $\tilde{J}^{\IH}$ (in the standard sense) to get limitedness of the derivatives of $\tilde{J}^{\IF}$ (and hence of $g^{\IF}$). The assumption of near-standardness could only be dropped in case that $\tilde{J}^{\IF}\approx\tilde{J}^{^*\IH}$ at all points in ${^*S}^2\times \IF/{^*\IC}$ and the norm of the $\ell$.th derivative of $\tilde{J}^{\IH}$ could be uniformly bounded at all points. Note that both assumptions would hold true when the flow $\phi^{\IH}_t$ could be \emph{uniformly} approximated by finite-dimensional flows. \end{remark}

\section{$J$-holomorphic spheres stay near-standard}

In this section we give the proof of the lemma \ref{stay-nearstandard} that makes precise the claim that the $\tilde{J}^{\IF}_t$-holomorphic sphere continues to stay infinitesimally close to $S^2\times\IH/\IC\subset{^*S}^2\times\IF/{^*\IC}$. Recall that in section 6 we claim that this follows, \emph{very informally speaking}, from the minimal surface property of $\tilde{J}^{\IF}_t$-holomorphic curves together with the fact that the induced Riemannian metric is of split-form $g^{\IF}_t\approx g^{\IH}_t\oplus \<\cdot,\cdot\>$ along $S^2\times \IH/\IC\subset {^*S}^2\times \IF/{^*\IC}$. In this section we make this argument precise using a non-standard version of Liouville's theorem. \\

We prove this crucial result in the following four steps. \\

\noindent\emph{Step 1: $\xi=\xi_t$ solves $(\CR + S^{\IF})\xi=\eta^{\IF}$.} \\

In order to apply Liouville's theorem, we need to bring the Cauchy-Riemann equation into play. Since in the standard finite-dimensional situation the corresponding result is well-known, see e.g. \cite{MDSa}, as already mentioned in the section on non-standard models we can simply again employ the transfer principle to obtain that tangent vector $\xi=\xi_t=\del u_t/\del t$ of the path $t\mapsto u_t$ of $\tilde{J}^{\IF}_t$-holomorphic spheres $u_t: {^*S}^2\to {^*S}^2\times \IF/{^*\IC}$ at a given fixed $t$ satisfies the linear Cauchy Riemann equation $$\nabla^{\IF} \xi + \tilde{J}^{\IF}(u)\cdot\nabla^{\IF}\xi \cdot i+\nabla^{\IF}_{\xi}\tilde{J}^{\IF}(u)\cdot Tu\cdot i \;=\; \nabla^{\IF}_t \tilde{J}^{\IF}(u)\cdot Tu\cdot i$$ with $u:=u_t$, $\tilde{J}^{\IF}:=\tilde{J}^{\IF}_t$ and $\nabla^{\IF}$ denotes the Levi-Civita connection with respect to the Riemannian metric $g^{\IF}={^*\omega}(\cdot,\tilde{J}^{\IF}\cdot)$ on ${^*S}^2\times\IF/{^*\IC}$. In order to keep the arguments simpler, we will actually use that, using a unitary trivialization $\Phi^{\IF}=\Phi^{\IF}_{t,u}: \IF\to u^*T({^*S}^2\times\IF/{^*\IC})$ satisfying $\tilde{J}^{\IF}(u)\cdot\Phi^{\IF}=\Phi^{\IF}\cdot J_0$, this linear Cauchy-Riemann equation can be rewritten in the form $$(\CR+S^{\IF})\xi\;=\;\eta^{\IF}$$ with 
\begin{eqnarray*} 
S^{\IF}&=&(\Phi^{\IF})^{-1}(\nabla^{\IF} \Phi^{\IF} + \tilde{J}^{\IF}(u)\cdot\nabla^{\IF}\Phi^{\IF} \cdot i+\nabla^{\IF}_{\Phi^{\IF}}\tilde{J}^{\IF}(u)\cdot Tu\cdot i), \\
\eta^{\IF}&=&(\Phi^{\IF})^{-1}(\nabla^{\IF}_t\tilde{J}^{\IF}(u)\cdot Tu\cdot i).
\end{eqnarray*} $ $\\

\noindent\emph{Step 2: There exists $T>0$ such that $S^{\IF}\xi_t$ and $\eta^{\IF}_t$ are near-standard for $t<T$.} \\

Define for every $n\in{^*\IN}$ the sequence of *-real numbers $$t_n:=\min\{t\in{^*[0,1]}: \|\xi_t\|_{1,p}\geq n\};$$ note that the minimum actually exists by transfer as the underlying set is again a set in the non-standard model by proposition \ref{internal-definition}. Because of the compactness of $[0,1]$ (in the standard sense), it follows that the monotone sequence of standard parts ${^{\circ}t_n}$, $n\in\IN$ converges to some standard $t_{\infty}\in [0,1]$. Then it directly follows from the definition that $\|\xi_t\|_{1,p}$ and hence also $\|\xi_t\|_{\infty}$ is limited for $t<t_{\infty}$ and we set $T:=t_{\infty}$. Using proposition \ref{limited-infinitesimal} we want to show that the limitedness of $\xi=\xi_t$ (together with the limitedness of the derivative of $u$) implies that $S^{\IF}\xi$ and $\eta^{\IF}$ are not only limited, but even near-standard in the sense that $(S^{\IF}\xi)(z)\in\IF$, $\eta^{\IF}(z)\in\IF$ are near-standard for all $z\in{^*S}^2$. For this let us again denote by $(e_i)_{i=1}^{\dim\IF}$ an orthonormal basis of $\IF$ extending a chosen complete orthonormal basis of $\IH$; further let us denote by $(e_i^*)_{i=1}^{\dim\IF}$ denote the corresponding dual basis of $\IF^*$. After introducing this notation, by proposition \ref{characterization} it is the goal to show that $$\sum_{i=2N+1}^{\dim\IF} \<S^{\IF}\xi,e_i\>^2\approx 0\;\;\textrm{and}\;\;\sum_{i=2N+1}^{\dim\IF} \<\eta^{\IF},e_i\>^2\approx 0$$ for all unlimited $\dim\IF>N\in{^*\IN}\backslash\IN$.\\

First, it follows from the smoothness of $H:\IH\to\IR$ that $$\sum_{i_1,\ldots,i_{k+1}=1}^{\infty} \<T^{k+1}H(^{\circ}u),e_{i_1}^*\otimes\ldots\otimes e_{i_{k+1}}^*\>^2<\infty$$ and hence, equivalently, $$\sum_{i_1,\ldots,i_k,j=1}^{\infty} \<T^k X^{\IH}(^{\circ}u),e_{i_1}^*\otimes\ldots\otimes e_{i_k}^*\otimes e_j\>^2<\infty.$$ Since $T^k\phi^{\IH}_0=0$ for all $k\geq 2$, this immediately gives $$\sum_{i_1,\ldots,i_k,j=1}^{\infty} \<T^k \phi^{\IH}_t(^{\circ}u),e_{i_1}^*\otimes\ldots\otimes e_{i_k}^*\otimes e_j\>^2<\infty$$ for all $k\geq 2$. Using that the almost complex structure $\tilde{J}^{\IH}$ as well as the canonical unitary trivialization $\Phi^{\IH}$ are defined using $\phi^{\IH}_t$, it follows that we in particular get \begin{eqnarray*} &&\sum_{i,j=1}^{\infty}\<\nabla^{\IH}\Phi^{\IH}(^{\circ}u),e_i^*\otimes e_j\>^2<\infty\;\;\textrm{and}\\&&\sum_{i_1,i_2,j=1}^{\infty} \<\nabla^{\IH}_{\Phi^{\IH}}\tilde{J}^{\IH}(^{\circ}u),e_{i_1}^*\otimes e_{i_2}^*\otimes e_j\><\infty.\end{eqnarray*}  

Combining proposition \ref{characterization} with the fact that $\nabla^{\IF}\Phi^{\IF}\approx\nabla^{\IH}\Phi^{\IH}$ and  $\nabla^{\IF}_{\Phi^{\IF}}\tilde{J}^{\IF}(u) \approx \nabla^{\IH}_{\Phi^{\IH}}\tilde{J}^{\IH}(^{\circ}u)$, it follows that \begin{eqnarray*} &&\sum_{j=2N+1}^{\dim\IF}\sum_{i=1}^{\dim\IF}\<\nabla^{\IF}\Phi^{\IF}(u),e_i^*\otimes e_j\>^2\approx 0\;\;\textrm{and}\\ &&\sum_{j=2N+1}^{\dim\IF}\sum_{i_1,i_2=1}^{\dim\IF} \<\nabla^{\IF}_{\Phi^{\IF}}\tilde{J}^{\IF}(u),e_{i_1}^*\otimes e_{i_2}^*\otimes e_j\>\approx 0\end{eqnarray*} for all unlimited $\dim\IF>N\in{^*\IN}\backslash\IN$. \\

But with this it follows with the Cauchy-Schwarz inequality and proposition \ref{limited-infinitesimal} that 
\begin{eqnarray*} 
&&\sum_{j=2N+1}^{\dim\IF}\<\nabla^{\IF}\Phi^{\IF}\cdot\xi,e_j\>^2\\&&\leq\; \Big(\sum_{j=2N+1}^{\dim\IF}\sum_{i=1}^{\dim\IF}\<\nabla^{\IF}\Phi^{\IF},e_i^*\otimes e_j\>^2\Big)\cdot\Big(\sum_{i=1}^{\dim\IF}\<\xi,e_i\>^2\Big)\approx 0,
\end{eqnarray*}
\begin{eqnarray*}
&&\sum_{j=2N+1}^{\dim\IF}\<\nabla^{\IF}_{\Phi^{\IF}\xi}\tilde{J}^{\IF}(u)\cdot Tu\cdot i,e_j\>^2 \\&&\leq\; 
\Big(\sum_{j=2N+1}^{\dim\IF}\sum_{i_1,i_2=1}^{\dim\IF} \<\nabla^{\IF}_{\Phi^{\IF}}\tilde{J}^{\IF}(u),e_{i_1}^*\otimes e_{i_2}^*\otimes e_j\>^2\Big)\cdot\|\xi\|^2\cdot\|Tu\|^2 \approx 0
\end{eqnarray*} 
for all unlimited $\dim\IF>N\in{^*\IN}\backslash\IN$, proving that $\nabla^{\IF}\Phi^{\IF}\cdot\xi$ and $\nabla^{\IF}_{\Phi^{\IF}\xi}\tilde{J}^{\IF}(u)\cdot Tu\cdot i$ are indeed near-standard by proposition \ref{characterization}. Note that by the same arguments we find that $\nabla^{\IF}_t \tilde{J}^{\IF}(u)\cdot Tu\cdot i$ is near-standard. Finally employing that $\tilde{J}^{\IF}(u)\approx\tilde{J}^{\IH}(^{\circ}u)$ and $\Phi^{\IF}\approx\Phi^{\IH}$ map near-standard elements to near-standard elements, it follows that $S^{\IF}\xi$ and $\eta^{\IF}$ are near-standard as desired. \\

\noindent\emph{Step 3: $\xi_t$ is near-standard for $t<T$.} \\

From the fact that $S^{\IF}\xi$ and $\eta^{\IF}$ are near-standard, it immediately follows from $(\CR+S^{\IF})\cdot\xi=\eta^{\IF}$ that $\CR\xi$ must be near-standard itself. We now complete the proof of the lemma by showing that the near-standardness of $\CR\xi$ implies the near-standardness of $\xi$.  \\

To this end, consider some unlimited $\dim\IF>N\in{^*\IN}\backslash\IN$. Defining $\xi_N^{\perp}(z) := \sum_{i=2N+1}^{\dim\IF} \<\xi(z),e_i\>\cdot e_i$, it immediately follows from $J_0 e_{2n+1}=e_{2n+2}$ that $$|\CR\xi_N^{\perp}(z)|^2 = \sum_{i=1}^{\dim\IF}\<\CR\xi_N^{\perp}(z),e_i\>^2=\sum_{i=2N_0+1}^{\dim\IF}\<\CR\xi(z),e_i\>^2\approx 0,$$ where the last identity follows from the characterization of near-standardness in proposition \ref{characterization}. \\

Denoting by $\|\cdot\|_p$, $\|\cdot\|_{\infty}$ the *-extensions of the $L^p$- and the supremum norm, respectively, it follows from transfer that $\|\CR\xi_N^{\perp}\|_p\leq (4\pi)^{1/p}\cdot\|\CR\xi_N^{\perp}\|_{\infty}\approx 0.$ On the other hand, it follows, after applying the transfer principle to the classical regularity estimate for the standard Cauchy-Riemann operator $\CR$, that $\|\xi_N^{\perp}\|_{1,p}\approx 0$, where $\|\cdot\|_{1,p}$ denotes the *-extension of the Sobolev $(1,p)$-norm. Note that here we additionally have to use the fact that $\xi(z_0)=\del\phi^{\IH}_t(0)/\del t\in\IH$ and hence $\xi_N^{\perp}(z_0)\approx 0$. Finally, after applying the transfer principle to the appropriate Sobolev embedding theorem, we indeed obtain $\|\xi_N^{\perp}\|_{\infty}\approx 0$, that is, $\xi_N^{\perp}(z)\approx 0$ for all $z\in{^*S^2}$, meaning that $\xi$ is near-standard.\\ 

Moreover, we see that our proof shows that $\xi_t$ and hence also the holomorphic spheres $u_t$ are not only near-standard in the sense of theorem \ref{holo-spheres} but even in a stronger Sobolev $H^{1,p}$-sense: Instead of just requiring that $\xi_N^{\perp}(z)$, $u^{\perp}_N(z)\approx 0$ for all $z\in{^*S^2}$, that is, $\xi^{\perp}_N$, $u^{\perp}_N\approx 0$ in the supremum norm, we actually get that $\xi^{\perp}_N$, $u^{\perp}_N\approx 0$ in the *-extension of the $H^{1,p}$-norm for all unlimited $\dim\IF>N\in{^*\IN}\backslash\IN$. This completes the proof of lemma \ref{stay-nearstandard}.

\section{Completing the proof}

It remains to be shown that there exists some standard $T>0$ such that $\xi_t(z)\in\IF$ is limited for all $z\in{^*S}^2$. Furthermore we have to insure that the near-standardness of the derivative $\xi_t=\del u_t/\del t$ of the path $t\mapsto u_t$ indeed proves the near-standardness of the path, that is, of $u_t$ for all $t<T$. Finally we show that the path $t\mapsto{^{\circ}u_t}$ is continuous in the $C^0$-sense and we prove the geometrical characterization of the maximal time $T$ in terms of pseudoholomorphic spheres as mentioned in theorem \ref{Main2}.\\

\noindent\emph{Step 1: A Lipschitz-type inequality}\\

The proof that the maximal time $T$ is a positive real number in the standard sense and the latter result about the nearstandardness property follow after having established an Lipschitz inequality for the map $(t,u)\to\xi_{t,u}=(D^{\IF}_{t,u})^{-1}\cdot\eta^{\IF}_{t,u}$ of the following form: There exist standard $\delta>0$ and $c_1,c_2>0$ such that $$\|\xi_{t,u}-\xi_{s,v}\|_{1,p}\,\leq\, c_1\,\|u-v\|_{1,p}\,+\, c_2\,|t-s|$$ whenever $\|u-u_0\|_{1,p}$, $\|v-u_0\|_{1,p}$, $|t|$, $|s|<\delta$ with the unique $\tilde{J}_0$-holomorphic sphere $u_0: S^2\to S^2\times \{0\}$.\\

In order to establish this result, it clearly suffices to prove the corresponding Lipschitz inequality for the maps $(t,u)\mapsto \eta^{\IF}_{t,u}$ and $(t,u)\mapsto S^{\IF}_{t,u}$, that is, there exist further standard constants $c_3$, $c_4$, $c_5$, $c_6>0$ such that $$\|\eta^{\IF}_{t,u}-\eta^{\IF}_{s,v}\|_p\,\leq\, c_3\,\|u-v\|_{1,p}\,+\, c_4\,|t-s|$$ and $$\|S^{\IF}_{t,u}-S^{\IF}_{s,v}\|\,\leq\, c_5\,\|u-v\|_{1,p}\,+\, c_6\,|t-s|$$ whenever $\|u-u_0\|_{1,p}$, $\|v-u_0\|_{1,p}$, $|t|$, $|s|<\delta$.\\

In order to prove the first of the latter two inequalities, note that 
\begin{eqnarray*} 
&&\|\eta^{\IF}_{t,u}-\eta^{\IF}_{s,v}\|_p \,\leq\, \|\nabla^{\IF}_t\tilde{J}^{\IF}_t(u)\cdot Tu-\nabla^{\IF}_t\tilde{J}^{\IF}_s(v)\cdot Tv\|_p \\ &&\leq\, \|\nabla^{\IF}_t\tilde{J}^{\IF}_t(u)\|_{\infty} \|Tu-Tv\|_p + \|\nabla^{\IF}_t\tilde{J}^{\IF}_t(u) -\nabla^{\IF}_t\tilde{J}^{\IF}_s(v)\|_{\infty} \|Tv\|_p. 
\end{eqnarray*}
Assuming that $\|u-u_0\|_{1,p}$, $\|v-u_0\|_{1,p}<\delta$ for some limited $\delta>0$, it first follows that from $\|Tu\|_p\leq \|Tu_0\|_p+\|u-u_0\|_{1,p}$ that $\|Tu\|$ and $\|Tv\|_p$ is limited as $\|Tu_0\|_p=(4\pi)^{1/p}$. Together with $\nabla^{\IF}_t\tilde{J}^{\IF}_t(u)\approx \nabla^{\IH}_t\tilde{J}^{\IH}_t({^{\circ}u})$ and $\nabla^{\IF}_t\tilde{J}^{\IF}_s(v)\approx \nabla^{\IH}_t\tilde{J}^{\IH}_s({^{\circ}v})$, it follows from proposition \ref{limited-infinitesimal} that the sum agrees, up to an infinitesimal error, with $$\|\nabla^{\IH}_t\tilde{J}^{\IH}_t({^{\circ}u})\|_{\infty} \|Tu-Tv\|_p + \|\nabla^{\IH}_t\tilde{J}^{\IH}_t({^{\circ}u}) -\nabla^{\IH}_t\tilde{J}^{\IH}_s({^{\circ}v})\|_{\infty} \|Tv\|_p $$ Since the map $(t,p)\mapsto \nabla^{\IH}_t\tilde{J}^{\IH}_t(p)$ is smooth in the standard sense, there exist some standard $\delta>0$ such that 
\begin{eqnarray*} 
&&\|\nabla^{\IH}_t\tilde{J}^{\IH}_t({^{\circ}u}) -\nabla^{\IH}_t\tilde{J}^{\IH}_s({^{\circ}v})\|_{\infty}\\ &&\leq\,2\|\nabla^{\IH}\nabla^{\IH}_t\tilde{J}^{\IH}_0(u_0)\|_{\infty}\cdot\|{^{\circ}u}-{^{\circ}v}\|_{\infty} + 2\|\nabla^{\IH}_t\nabla^{\IH}_t\tilde{J}^{\IH}_0(u_0)\|_{\infty}\cdot |t-s|
\end{eqnarray*}  
whenever $\|{^{\circ}u}-u_0\|_{\infty}$, $\|{^{\circ}v}-u_0\|_{\infty}$, $|t|$, $|s|<\delta$. Note the constants $\|\nabla^{\IH}\nabla^{\IH}_t\tilde{J}^{\IH}_0(u_0)\|_{\infty}$ and $\|\nabla^{\IH}_t\nabla^{\IH}_t\tilde{J}^{\IH}_0(u_0)\|_{\infty}$ only depend on the Hamiltonian $H$ and its derivatives near the \emph{compact} set $B^2(R)\times\{0\}\subset Z^{\IH}(R)\subset\IH$; in particular they are finite and hence limited. Together with $$\|{^{\circ}u}-u_0\|_{\infty}\approx\|u-u_0\|_{\infty}\leq c\|u-u_0\|_{1,p}$$ it follows that 
\begin{eqnarray*} 
&&\|\nabla^{\IH}_t\tilde{J}^{\IH}_t({^{\circ}u})\|_{\infty} \|Tu-Tv\|_p + \|\nabla^{\IH}_t\tilde{J}^{\IH}_t({^{\circ}u}) -\nabla^{\IH}_t\tilde{J}^{\IH}_s({^{\circ}v})\|_{\infty} \|Tv\|_p\\ &&\leq\, 2  \|\nabla^{\IH}_t\tilde{J}^{\IH}_0(u_0)\|_{\infty}\cdot \|u-v\|_{1,p} \\&&+\, 2\|\nabla^{\IH}\nabla^{\IH}_t\tilde{J}^{\IH}(u_0)\|_{\infty}\cdot c\cdot \|u-v\|_{1,p}\cdot 2\|Tu_0\|_p\\&&+\, 2\|\nabla^{\IH}_t\nabla^{\IH}_t\tilde{J}^{\IH}(u_0)\|_{\infty}\cdot |t-s|\cdot 2\|Tu_0\|\\&&\leq\; c_3\cdot\|u-v\|_{1,p}\,+\, c_4\cdot |t-s|
\end{eqnarray*}
whenever $\|u-u_0\|_{1,p}$, $\|v-u_0\|_{1,p}$, $|t|$, $|s|<\delta$, possibly after making $\delta>0$ smaller but still standard, where the constant $c_3$, $c_4>0$ are defined as 
\begin{eqnarray*}
c_3 &:=& 2 \|\nabla^{\IH}_t\tilde{J}^{\IH}_0(u_0)\|_{\infty}+2\|\nabla^{\IH}\nabla^{\IH}_t\tilde{J}^{\IH}_0(u_0)\|_{\infty}\cdot c\cdot 2\cdot (4\pi)^{1/p},\\ c_4&:=& 2\|\nabla^{\IH}_t\nabla^{\IH}_t\tilde{J}^{\IH}(u_0)\|_{\infty}\cdot  2\cdot (4\pi)^{1/p}.
\end{eqnarray*}
In a similar fashion, one can prove that the constants $c_5$, $c_6>0$ are given by 
\begin{eqnarray*}
c_5 &:=& 2\|\nabla^{\IH}\nabla^{\IH}\Phi^{\IH}_0\|_{\infty} \cdot c+2 \|\nabla^{\IH}\tilde{J}^{\IH}_0(u_0)\|_{\infty}+2\|\nabla^{\IH}\nabla^{\IH}\tilde{J}^{\IH}_0(u_0)\|_{\infty}\cdot c\cdot 2\cdot (4\pi)^{1/p},\\ c_6&:=& 2\|\nabla^{\IH}_t\nabla^{\IH}\Phi^{\IH}_0\|_{\infty} +2\|\nabla^{\IH}_t\nabla^{\IH}\tilde{J}^{\IH}_0(u_0)\|_{\infty}\cdot  2\cdot (4\pi)^{1/p}.
\end{eqnarray*}
In particular, since also the constants $\|\nabla^{\IH}\nabla^{\IH}\tilde{J}^{\IH}(u_0)\|_{\infty}$ and $\|\nabla^{\IH}_t\nabla^{\IH}\tilde{J}^{\IH}(u_0)\|_{\infty}$ as well as $\|\nabla^{\IH}\nabla^{\IH}\Phi^{\IH}_0\|_{\infty}$ and $\|\nabla^{\IH}_t\nabla^{\IH}\Phi^{\IH}_0\|_{\infty}$ only depend on the Hamiltonian $H$ and its derivatives near the \emph{compact} set $B^2(R)\times\{0\}\subset Z^{\IH}(R)\subset\IH$, the finiteness of all constants $c_3,c_4,c_5,c_6$ and hence also $c_1,c_2$ does \emph{not} require extra uniform bounds on the derivatives of $H$. \\

\noindent\emph{Step 2: $C^0$-continuity and geometrical characterization of the maximal time $T$}\\

After we have established the existence of a one-dimensional family $t\mapsto {^{\circ}u_t}$ of $\tilde{J}^{\IH}_t$-holomorphic spheres, $t\in [0,T)$, it first remains to prove that the family is indeed continous in the standard sense. As in step 2 in section 6 we observe that the map $t\mapsto u_t$ is differentiable in the non-standard sense with limited derivative $\xi_t=\del u_t/\del t$. More precisely, it follows that $\|u_{t'}-u_t\|_{1,p}\leq n |t'-t|$ as long as $t',t\in [0,t_n]$ for all $n\in\IN$. After applying the transfer principle to the Sobolev embedding theorem, it follows that for every $n\in\IN$ there exists some standard constant $d_n>0$ such that for the *-extension of the supremum norm we get $\|u_{t'}-u_t\|_{\infty}\leq d_n |t'-t|$ and hence $\|{^{\circ}u}_{t'}-{^{\circ}u}_t\|_{\infty}\leq d_n |t'-t|$, that is, the family $t\mapsto u_t$ is continuous in the $C^0$-sense. \\

It remains to prove the geometrical characterization of the maximal time $T$ given in theorem \ref{Main2}. For this consider the sequence $t_n$, $n\in{^*\IN}$ defined in step 2 of the last section and denote by $u_n=u_{t_n}$, $n\in{^*\IN}$ the corresponding sequence of $\tilde{J}^{\IF}_{t_n}$-holomorphic spheres in ${^*S^2}\times\IF/{^*\IC}$. Following the given two alternatives in theorem \ref{Main2}, we may assume that the corresponding standard sequence ${^{\circ}u_n}$, $n\in\IN$ of $\tilde{J}^{\IH}_{^{\circ}t_n}$-holomorphic spheres has a $C^0$-convergent subsequence converging to a map $u_{\infty}: S^2\to S^2\times \IH/\IC$. \\

By saturation it follows that there must exist some unlimited $N\in{^*\IN}\backslash\IN$ such that $t_N\approx t_{\infty}=T$ and $u_N\approx u_{\infty}$: Indeed, define $$A_{\eps}=\{n\in{^*\IN}: \|u_n-u_{\infty}\|_{\infty},|t_n-t_{\infty}|<\eps\}.$$ Since for any finite collection $\eps_1,\ldots,\eps_k\in\IR^+$ we know that $A_{\eps_1}\cap\ldots\cap A_{\eps_k}\neq\emptyset$, it follows by the saturation principle that there exists $N\in\bigcap_{\eps} A_{\eps}\neq\emptyset$. Since $u_N$ is a $\tilde{J}^{\IF}_{t_N}$-holomorphic sphere and near-standard, it follows from lemma \ref{no-bubbling} that $u_{\infty}={^{\circ}u_N}$ is a $J^{\IH}_T$-holomorphic sphere in the standard sense. \emph{In particular, it follows from our non-standard bubbling-off argument that the $C^0$-limit $u_{\infty}$ is indeed smooth in the standard sense.} Furthermore, by step 2 in the last section we have that $\eta^{\IF}_{t_N}$ is nearstandard and hence limited. On the other hand, since $\|\xi_{t_N}\|_{1,p}=N$ is unlimited, it follows that the operator norm of the inverse operator $(D^{\IF}_{t_N})^{-1}$ of the linearized Cauchy-Riemann operator $D^{\IF}=D^{\IF}_{t_N}$ at $u_N$ must be unlimited. Since $D^{\IF}=\CR+S^{\IF}\approx \CR+S^{\IH}_{t_{\infty}}=D^{\IH}$ with 
\begin{eqnarray*} 
S^{\IF}&=&(\Phi^{\IF})^{-1}(\nabla^{\IF} \Phi^{\IF} + \tilde{J}^{\IF}(u_N)\cdot\nabla^{\IF}\Phi^{\IF} \cdot i+\nabla^{\IF}_{\Phi^{\IF}}\tilde{J}^{\IF}(u_N)\cdot Tu_N\cdot i), \\
S^{\IH}&=&(\Phi^{\IH})^{-1}(\nabla^{\IH} \Phi^{\IH} + \tilde{J}^{\IH}(u_{\infty})\cdot\nabla^{\IH}\Phi^{\IH} \cdot i+\nabla^{\IH}_{\Phi^{\IH}}\tilde{J}^{\IH}(u_{\infty})\cdot Tu_{\infty}\cdot i),
\end{eqnarray*}
it follows that the linearization of the Cauchy-Riemann operator at $u_{\infty}$ cannot have a bounded inverse.

\end{document}